\documentclass[12pt]{amsart}
\usepackage{amsfonts}
\usepackage{amsmath}
\usepackage{amsfonts}
\usepackage{amssymb}
\usepackage{amscd}

\newtheorem{theorem}{Theorem}[section]

\newtheorem{lemma}[theorem]{Lemma}
\newtheorem{proposition}[theorem]{Proposition}

\newtheorem{remark}{Remark}[section]

\newtheorem{definition}{Definition}[section]
\newtheorem*{definition*}{Definition}

\newtheorem*{Theorem A}{Theorem A}
\newtheorem*{Theorem B}{Theorem B}
\newtheorem*{Theorem C}{Theorem C}
\newtheorem*{Corollary A1}{Corollary A1}
\newtheorem*{Corollary A2}{Corollary A2}
\newtheorem*{Corollary C1}{Corollary C1}
\newtheorem*{lemma*}{Lemma}
\newtheorem*{proposition*}{Proposition}
\newtheorem*{theorem*}{Theorem}

\theoremstyle{definition} \theoremstyle{remark}
\numberwithin{equation}{section}

\newcommand{\rmi}{(\rm i)\,}
\newcommand{\rmii}{(\rm ii)\,}
\newcommand{\rmiii}{(\rm iii)\,}

\newcommand{\vp}{\varphi}
\newcommand{\diver}{{\mathop{\mathrm div}\,}}
\renewcommand{\div}[1]{{\mathop{\mathrm div}}\left(#1\right)}

\newcommand{\nablaphi}[1]{\vert \nabla #1\vert^{-1}
\varphi(\vert \nabla #1 \vert)\nabla #1}
\newcommand{\modnabla}[1]{\vert \nabla #1\vert }

\newcommand{\vol } {\mathrm{vol}\,}
\newcommand{\volD } {\mathrm{vol}_D\,}

\renewcommand{\d}{\delta}
\newcommand{\ty}{\infty}

\newcommand{\bdr}{\partial}
\newcommand{\R}{\mathbb{R}}

\newcommand{\ds}{\displaystyle}

\begin{document}

\author{Luciano Mari}
\address{Dipartimento di Matematica\\
Universit\`a di Milano\\
via Saldini 50\\
I-20133 Milano, ITALY} \email{luciano.mari@unimi.it}

\author{Marco Rigoli}
\address{Dipartimento di Matematica\\
Universit\`a di Milano\\
via Saldini 50\\
I-20133 Milano, ITALY} \email{rigoli@mat.unimi.it}

\author{Alberto G. Setti}
\address{Dipartimento di Fisica e Matematica
\\
Universit\`a dell'Insubria - Como\\
via Valleggio 11\\
I-22100 Como, ITALY}
 \email{alberto.setti@uninsubria.it}

\subjclass[2000]{58J05, 35J60}

\keywords{Keller--Osserman condition, diffusion type operators, weak maximum principles,
weighted Riemannian manifolds, quasi-linear elliptic inequalities}

\title[Diffusion type operators]
{
Keller--Osserman conditions for
diffusion-type operators on
Riemannian manifolds}

\begin{abstract}
In this paper we obtain generalized Keller-Osserman conditions for wide classes
of differential inequalities on weighted Riemannian manifolds of the form
$L u\geq b(x) f(u) \ell(|\nabla u|)$ and $L u\geq b(x) f(u) \ell(|\nabla u|)
- g(u) h(|\nabla u|)$, where $L$ is a non-linear diffusion-type operator.
Prototypical examples of these operators are the $p$-Laplacian and
the mean curvature operator. While we concentrate on non-existence results,
in many instances the conditions we describe are in fact necessary for non-existence.
The geometry of the underlying manifold does not affect the form of the Keller-Osserman
conditions, but is reflected, via bounds for the modified Bakry-Emery Ricci curvature,
by growth conditions for the functions $b$ and  $\ell$. We also describe a weak
maximum principle related to inequalities of the above form which
extends and improves previous results valid for the
$\vp$-Laplacian.
\end{abstract}
 \maketitle

\section{Introduction}
\label{intro} Consider the Poisson-type inequality on Euclidean
space $\R^m$
\begin{equation}
\label{PoissonRm}
\Delta u \geq f(u)
\end{equation}
where $f\in C^0\bigl([0,+\infty)\bigr)$, $f(0) = 0$ and $f(t)>0$ if
$t>0$. By an entire solution of (\ref{PoissonRm}) we mean a $C^1$
function $u$ satisfying (\ref{PoissonRm}) on $\R^m$ in the sense of
distributions. Let
\begin{equation}
\label{defF}
F(t) = \int_0^t f(s) \,ds.
\end{equation}
It is well know that if $f$ satisfies the Keller--Osserman condition
\begin{equation}
\label{KellerOssermanRm} \frac 1{\sqrt{F(t)}} \in L^1(+\infty),
\end{equation}
then (\ref{PoissonRm}) has no nonnegative entire solutions except
$u\equiv 0$. Note that in the case where $f(t)=t^q$ the
integrability condition expressed by (\ref{KellerOssermanRm}) is
equivalent to $q>1$. But (\ref{KellerOssermanRm}) is sharper than the
condition on powers it is implied by. For instance
(\ref{KellerOssermanRm}) holds if $f(t)=t\log^\beta (1+t)$ with
$\beta >2$.

As a matter of fact, if the Keller--Osserman condition fails, that is, if
\begin{equation}
\label{noKellerOssermanRm}
\frac 1{\sqrt{F(t)}} \not\in L^1(+\infty),
\end{equation}
 then
inequality (\ref{PoissonRm}) admits positive solutions. Indeed,
consider the ODE problem
\begin{equation}
\label{radialPoissonRm}
\begin{cases}
\alpha'' + \frac{m-1}r \alpha' = f(\alpha) &\\
\alpha(0)=\alpha_o>0,\,\, \alpha'(0)=0.
\end{cases}
\end{equation}
General theory yields the existence of a solution in a maximal
interval $[0,R)$ and a first integration of (\ref{radialPoissonRm})
gives $\alpha'>0$ on $(0,R)$. Suppose by contradiction that
$R<+\infty$. Using the maximality condition and the monotonicity of
$\alpha$ we obtain
\begin{equation}
\label{lim_alpha}
\lim_{r\to R^-}\alpha(r)=+\infty.
\end{equation}
On the other hand it follows from (\ref{radialPoissonRm}) that
\begin{equation*}
\alpha'\alpha'' \leq f(\alpha) \alpha',
\end{equation*}
whence integrating over $[0,r]$, $0<r\leq R$, changing variables in
the resulting integral, and taking square roots we obtain
\begin{equation*}
\frac{\alpha'}{\sqrt{F(\alpha)}}\leq \sqrt 2.
\end{equation*}
A further integration over $[0,r]$ with $0<a<r<R$ yields
\begin{equation*}
\int_{\alpha(a)}^{\alpha(r)} \frac{dt}{\sqrt{F(t)}}
\leq \sqrt 2 (r-a)
\end{equation*}
and letting $r\to R^-$ and using (\ref{lim_alpha}) we contradict
(\ref{noKellerOssermanRm}). This shows that the function $\alpha$
is defined on $[0,+\infty)$. Setting $u(x)= \alpha(r(x))$ ($r(x)=|x|$)
gives rise to a radial positive entire solution of (\ref{PoissonRm}).
Note however that any nonnegative solution of (\ref{PoissonRm})
must diverge at infinity sufficiently fast.
Indeed, it follows from \cite{PigolaRigoliSetti-MatCont}, Corollary 16,  that if
$u\geq 0$ is an entire solution of (\ref{PoissonRm}) satisfying
\begin{equation*}
u(x)=o\bigl(r(x)^\sigma\bigr) \text{ as } r(x) \to +\infty,
\end{equation*}
with $0\leq \sigma<2$, and $f$ is non-decreasing, then $u\equiv 0$.
Note that this latter conclusion can be hardly deduced from
(\ref{noKellerOssermanRm}).

We also observe that differential inequalities of the type (\ref{PoissonRm})
often appear in connection with geometrical problems on complete manifolds
and, in fact, R. Osserman introduced condition
(\ref{KellerOssermanRm}) in \cite{Osserman} in his investigation on the type of
a Riemann surface. For a number of further examples we refer,  for instance, to
\cite{PigolaRigoliSetti-Memoirs}.

Motivated by the above considerations, from now on we will denote
with $(M,\langle \, ,\rangle)$ a complete, non-compact, connected
Riemannian manifold of dimension $m\geq 2$. We fix an origin $o$ in
$M$ and we let $r(x)=\mathrm{dist} (x,o)$ be the Riemannian distance
from the chosen reference point, and we denote by
$B_r$ the geodesic ball of radius $r$ centered at $o$ and with
$\partial B_r$ its boundary.

Given a a positive function $D(x)\in C^2(M)$ and  a non-negative
function $\vp\in C^0(\R^+_0)\cap C^1(\R^+)$, where, as usual
$\R^+=(0,+\infty)$ and $\R^+_0= [0,+\infty)$,
we consider the diffusion-type operator defined on $M$ by the
formula
\begin{equation*}
L_{D,\vp} u= \frac 1{D} \mathop{div} \bigl(D |\nabla u|^{-1}
\vp (|\nabla u|) \nabla u \bigr).
\end{equation*}
For instance, if $D\equiv 1$ and $\vp (t) = t^{p-1}$, $p>1$, or
$\vp(t) = \frac t{\sqrt{1+t^2}}$ we recover the usual $p$-Laplacian and
 the mean curvature operator, respectively.

If $b(x)\in C^0(M)$ and $\ell\in C^0(\R^+_0)$, we will be interested in
solutions of the differential inequality
\begin{equation}
\label{main_ineq}
L_{D,\vp} u\geq b(x) f(u)\ell(|\nabla u|).
\end{equation}
By an entire classical weak solution of (\ref{main_ineq}) we mean a
$C^1$ function $u$ on $M$ which satisfies the inequality in the sense
of distributions, namely,
\begin{equation}
\label{weak_ineq}
-\int |\nabla u|^{-1} \vp(|\nabla u|)\langle \nabla u,\nabla \vp\rangle
D\,dV   \geq \int b(x) f(u) \ell(|\nabla u|) \psi D\,dV
\end{equation}
for every non-negative function $\psi \in C^\infty_c(M)$, where
we have denoted with $dV$ the Riemannian volume element.

Since we are dealing with a diffusion-type operator, the interplay
between analysis and geometry will be taken into account by means of
the modified Bakry--Emery Ricci tensor that we now introduce. Following
Z.~Qian (\cite{Qian-EstimatesWeightedVolume}), for $n>m$
let
\begin{equation}
\label{modified_Bakry-Emery_Ricci}
\begin{split}
\mathrm{Ricc}_{m,n} (L_D)& = Ricc_M - \frac 1 D \mathop{Hess} D +
\frac{n-m-1}{n-m} \frac 1 {D^2} dD\otimes dD\\
& = \mathrm{Ricc}(L_D) - \frac 1{n-m} \frac 1 {D^2} dD\otimes dD
\end{split}
\end{equation}
be the modified Bakry--Emery Ricci tensor, where
$\mathrm{Ricc}(L_D)$ is the usual Bakry--Emery Ricci tensor,
$\mathrm{Ricc}_M$ is the Ricci tensor of $(M,\langle\, ,\rangle)$,
(see D.~Bakry and P.~Emery, \cite{Bakry-Emery_DiffusionsHypercontractives}),
and where, to simplify notation, we have denoted with  $L_D$ the operator
$L_{D,\vp}$ for $\vp(t)=t$.

We introduce some more terminology.

\begin{definition}
\label{Cincreasing_def}
Let $g$ be a real valued function defined on $\R^+$. We say that $g$
is $C$-increasing on $\R^+$ if there exists  a constant $C\geq 1$ such
that
\begin{equation}
\label{Cincreasing}
\sup_{s\in(0,t]} g(s)\leq C g(t)\quad \forall t\in \R^+.
\end{equation}
\end{definition}
It is easily verified that the above condition is equivalent to
\begin{equation*}
\inf_{s\in [t, +\infty)} g(s)\geq \frac 1 C g(t)\quad \forall t\in
\R^+,
\end{equation*}
and both formulations will be used in the sequel. Clearly,
(\ref{Cincreasing}) is satisfied with $C=1$ if $g$ is non-decreasing
on $\R^+$. In general, the validity of (\ref{Cincreasing}) allows a
controlled oscillatory behavior such as, for instance, that of
$g(t) = t^2(2+\sin t)$.

In order to state our next result, we introduce the following set of
assumptions.
\begin{itemize}
\item[($\Phi_0$)] $\vp'>0$ on $\R^+$.
\item[($F_1$)] $f\in C(\R)$,
$f(0)=0$, $f(t)>0$ if $t>0$ and  $f$  is $C$-increasing on $\R^+$.
\item[($L_1$)] $\ell\in C^0(\R^+_0)$, $\ell(t)>0$ on $\R^+$.
\item[($L_2$)] $\ell$ is $C$-increasing on $\R^+$.
\item[($\vp\ell$)] $\liminf_{t\to 0^+} \frac{\vp(t)}{\ell(t)} = 0$,
\, $\frac{t\vp'(t)}{\ell(t)}\in L^1(0^+)\setminus L^1(+\infty)$.
\item[($\theta$)] there exists $\theta\in \R$ such that the
functions
$$
t\to\frac{\vp'(t)}{\ell(t)}t^\theta \quad \text{ and } \quad
t\to \frac{\vp(t)}{\ell(t)} t^{\theta -1}
$$
are $C$-increasing on $\R^+$.
\end{itemize}

Clearly the last two conditions relate the operator $L_{D,\vp}$ to
the gradient term $\ell$, and, in general,  they are not
independent. As we shall see below, in favorable circumstances
($\theta$) implies ($\vp\ell$). This is the case, for instance, in
the next Theorem~A when $\theta<1.$ For a better understanding of
these two assumptions, we examine the special but important case
where $\ell(t)= t^q,$ $q\geq 0.$ First we consider the case of the
$p$-Laplacian, so that $\vp(t)= t^{p-1}$, $p>1$. Then, given
$\theta\in \R,$ ($\vp\ell$) and ($\theta$) are simultaneously
satisfied provided
\begin{equation*}
p>q+1\quad\text{ and } \quad \theta\geq q-p+2.
\end{equation*}
If we consider $\vp(t)= t e^{t^2}$ (which, when $D\equiv 1$, gives
rise to the operator associated to the exponentially harmonic
functions, see \cite{DucEells-ExpHarmonicFcns} and
\cite{EellsLemaire-ExpHarmonicFcns}), then
($\vp\ell$) and ($\theta$) are both satisfied provided
\begin{equation*}
q<1 \quad \text{ and }\quad q \leq \theta.
\end{equation*}
If $\vp=\frac{t}{\sqrt{1+t^2}}$, which, for $D\equiv 1$,
corresponds to the ``mean curvature operator'', then ($\vp\ell$)
does not hold for any $q\geq 0.$  However, a variant of our arguments will allow
us to analyze this situation, see Section 4 below.

Because of ($L_1$) and ($\vp\ell$) we may define a
$C^1$-diffeomorphism $K:\R_0^+\to \R_0^+$ by the formula
\begin{equation}
\label{K_def}
K(t) = \int_0^t \frac{s\vp'(s)}{\ell(s)} ds.
\end{equation}
Since $K$ is increasing on $\R_0^+$ so is its inverse  $K^{-1}$.
Moreover, when  $\ell\equiv 1$ then
\begin{equation*}
K'(t) = \widetilde H'(t)
\end{equation*}
where
\begin{equation*}
\widetilde  H(t) = t\vp(t) - \int_0^t \vp(s) ds
\end{equation*}
is the pre-Legendre transform of $t\to \int_o^t \vp(s) ds.$

Having defined $F$ as in (\ref{defF}) we are ready to introduce our
first generalized Keller--Osserman condition.
\begin{equation}
\label{KO}\tag{KO}
\frac 1{K^{-1}(F(t))}\in L^1(+\infty).
\end{equation}
It is clear that, in the case of the Laplace--Beltrami operator (or
more generally, of the $p$-Laplacian) and for $\ell\equiv 1,$ (KO)
is equivalent to the classical Keller--Osserman condition
(\ref{KellerOssermanRm}).
After this preparation we are ready to state
\begin{Theorem A}
\label{Theorem A}
Let $(M,\langle\,,\rangle)$ be a complete manifold satisfying
\begin{equation}
\label{genRicci_lower_bound}
\mathrm{Ricc}_{n,m} (L_D) \geq H^2(1+r^2)^{\beta/2},
\end{equation}
for some $n>m$, $H>0$ and $\beta\geq -2$. Let also $b(x)\in C^0(M)$
be a non-negative function such that
\begin{equation}
\label{b_lower_bound}
b(x)\geq \frac C { r(x)^{\mu}} \quad \text{if } \, r(x)\gg 1,
\end{equation}
for some $C>0$ and $\mu\geq 0.$  Assume that ($\Phi_0$), ($F_1$),
($L_1$), ($L_2$), ($\vp\ell$), ($\theta$) and (\ref{KO}) hold, and
suppose that
\begin{equation}
\label{thetabetamu}\tag{$\theta\beta\mu$}
\begin{cases}
\theta < 1-\beta/2 - \mu \,\, \text{ or } \,\,\theta = 1-\beta/2 -
\mu<1 &\text{if }\,  \mu>0\\
\theta< 1-\beta/2 &\text{ if } \, \mu=0.
\end{cases}
\end{equation}
Then any entire classical
weak solution $u$ of the differential inequality (\ref{main_ineq})
is either non-positive or constant. Furthermore, if $u\geq 0$ and $\ell(0) >0$,
then $u\equiv 0$.
\end{Theorem A}
We remark that letting $\beta<-2$ in (\ref{genRicci_lower_bound})
yields the same estimates valid for $\beta=-2,$ which roughly correspond to the
Euclidean behavior. Correspondingly, the conclusion of Theorem A is not improved
by  such a strengthening of the assumption on the modified Bakry--Emery Ricci curvature.

To better appreciate the result and the role played by geometry, we
state the following consequence for the $p$-Laplace operator
$\Delta_p$.

\begin{Corollary A1}
\label{Corollary A1}
Let $(M,\langle\,,\rangle)$ and $b(x)$ be as in the statement of
Theorem A and satisfying (\ref{genRicci_lower_bound}) with $D\equiv 1$ (so that
$\mathrm{Ricc}_{n,m}= \mathrm{Ricc}$) and (\ref{b_lower_bound}).
Let $f$ satisfy ($F_1$) and let $\ell(t)= t^q$,
for some $q\geq 0$. Assume that $p$ and  $\mu$ satisfy
\begin{equation*}
p>q+1,\quad 0\leq \mu\leq p-q,\quad \beta \leq 2(p-q-\mu-1).
\end{equation*}
If
\begin{equation*}
\tag{KO}
\frac 1{F(t)^{1/(p-q)}}\in L^1(+\infty),
\end{equation*}
then any  entire classical weak solution $u$ of the
differential inequality
\begin{equation*}
\Delta_p u \geq b(x) f(u) |\nabla u|^q
\end{equation*}
is either non-positive or
constant.
\end{Corollary A1}
Note that if $p=2$ and $q=\mu=0$, then the maximum amount of
negative curvature allowed is obtained by choosing $\beta=2$.
In particular, the result covers the cases of Euclidean and hyperbolic
space. We observe in passing that the choice $\beta=2$ is borderline
for the stochastic completeness of the underlying manifold.

To include in our analysis the case of the mean curvature operator
we state the following consequence of Theorem~4.1.

\begin{Corollary A2}
\label{Corollary A2}
Let $(M,\langle\,,\rangle)$ and $b(x)$ be as in the statement of
Theorem A and satisfying (\ref{genRicci_lower_bound}) with $D\equiv 1$
and (\ref{b_lower_bound}). Let $f$ satisfy ($F_1$) and let $\ell(t)= t^q$,
for some $q\geq 0$. Assume   $\mu\geq 0$ and that
\begin{equation*}
0\leq q< -\frac \beta 2 -\mu.
\end{equation*}
If
\begin{equation*}
\tag{$\rm{\widehat{KO}}$}
\frac 1{F(t)^{1/(1-q)}}\in L^1(+\infty),
\end{equation*}
then any non-negative, entire classical weak solution $u$ of the
differential inequality
\begin{equation*}
\div{\frac{\nabla u }{\sqrt{1+\modnabla u^2}}} \geq b(x) f(u) |\nabla u|^q
\end{equation*}
is constant.
\end{Corollary A2}
Note that, contrary to Corollary A1,  the case of hyperbolic space, which corresponds
to $\beta=0,$ is not covered by Corollary A2. On the other hand, if
$\beta=-2,$ which, as already mentioned, roughly corresponds to a
Euclidean behavior, the conditions on the parameters become
\begin{equation*}
\mu\geq 0, \quad 0\leq q<1-\mu,
\end{equation*}
and they are clearly compatible. This is one of the instances where
the interaction between geometry and differential operators comes into
play.

As briefly remarked at the beginning of this introduction, the
failure of the Keller--Osserman condition may yield existence of
non-constant non-negative entire solutions. The next result shows
that such solutions, if they exist, have to go to infinity
sufficiently fast depending on the geometry of $M$ and, of course,
of the relevant parameters in the differential inequality satisfied.
To state our result we introduce the following set of assumptions.

\begin{itemize}
\item[($\Phi_1$)] \rmi $\vp(0) = 0$; \rmii $\vp(t)\leq At^\delta$ on $\R^+$, for some
$A$, $\delta>0$.
\item[($F_0$)] $f\in C^0(\R^+_0)$.
\item[($L_3$)] $\ell\in C^0(\R^+_0)$, $\ell(t)\geq C t^\chi$ on $\R^+$, for some
$C>0$, $\chi\geq 0$.
\item[($b_1$)] $b\in C^0(M)$, $b(x)>0$ on $M$, $b(x)\geq \frac
C{r(x)^\mu}$ if $r(x)\gg 1$, for some $C>0,$ $\mu\in \R.$
\end{itemize}
\begin{Theorem B}
\label{Theorem B}
Let $(M, \langle\, ,\rangle)$ be a  complete Riemannian manifold,
and assume that conditions ($\Phi_1$), ($F_0$), ($L_3$) and ($b_1$) hold.
Given $\sigma\geq 0$, let $\eta= \mu - (1+\delta - \chi)(1-\sigma)$
and suppose that
\begin{equation*}
\sigma \geq \eta,\quad 0\leq \chi< \delta.
\end{equation*}
Let $u$ be a non-constant entire classical weak solution of
\begin{equation}
\tag{\ref{main_ineq}}
L_{D,\vp} u \geq b(x) f(u) \ell(\modnabla u),
\end{equation}
and suppose that either
\begin{equation}
\label{u_growth}
\begin{split}
&\sigma >0,\qquad\liminf_{t\to +\infty} f(t)>0 \,\,\text{ and}\\
&u_+(x)= \max\{ u(x), 0\} = o\bigl(r(x)^\sigma\bigr)
\quad \text{ as }\, r(x)\to +\infty,
\end{split}
\end{equation}
or
\begin{equation}
\sigma=0 \quad \text{ and }\quad  u^*= \sup_M u<+\infty.
\end{equation}
Assume further that either
\begin{equation}
\label{vol_growth_exp}
\liminf_{r\to +\infty} \frac{\log \int_{ B_r} D(x) dV(x)}{r^{\sigma -\eta}} <+\infty
\quad \text{ if } \sigma -\eta >0
\end{equation}
or
\begin{equation}
\label{vol_growth_poly}
\liminf_{r\to +\infty} \frac{\log \int_{B_r} D(x) dV(x)}{\log r} <+\infty
\quad \text{ if } \sigma -\eta =0.
\end{equation}
Then $u^*<+\infty$ and $f(u^*)\leq 0$. In particular, if we also assume that
$f(t)>0$ for $t>0$, and that  $u(x_o)>0$ for some $x\in M$, then  $u$ is
constant on $M$, and if  in addition  $f(0)=0$ and $\ell(0)>0$, then $u\equiv 0$ on $M$.
\end{Theorem B}

Observe that the growth condition (\ref{u_growth}) is sharp.
Indeed, we consider  the case of the $p$-Laplace operator on Euclidean
space, for which $D\equiv 1$ and $\delta=p-1$, and suppose that $\chi=\mu=0$ and
$\sigma=\eta$. Since $\eta=p(\sigma-1)$, the latter condition amounts to $\sigma
=p',$ the H\"older conjugate exponent of $p$. Since condition
(\ref{vol_growth_poly}), which now reads
\begin{equation*}
\liminf_{r\to +\infty} \frac{\log \vol B_r}{\log r}
<+\infty,
\end{equation*}
is clearly satisfied, 
all assumptions of Theorem~B hold. On the other hand, a simple
computation shows that the function $u(x)= \frac 1 {p'} r(x)^{p'}$
is a classical entire weak solution of $\Delta_p u=m$, for which
(\ref{u_growth}) barely fails to be met.

We also stress that while in Theorem~A the main geometric assumption
is the radial lower bound on the modified Bakry--Emery Ricci
curvature expressed by (\ref{genRicci_lower_bound}), in Theorem ~B
we consider either (\ref{vol_growth_exp}) or
(\ref{vol_growth_poly}), which we interpret as follows. Let
$dV_D= D dV$ be the measure with density $D(x)$, so that, for every
measurable set $\Omega$,
\begin{equation*}
\vol_D(\Omega) = \int_\Omega D(x)\,dV,
\end{equation*}
and consider the weighted Riemannian manifold $(M, \langle ,\,\rangle,
dV_D)$.
With this notation, we may rewrite, for instance
(\ref{vol_growth_exp}), in the form
\begin{equation}
\label{vol_growth_exp_bis}
\liminf_{r\to +\infty} \frac{\log \vol_D B_r}{r^{\sigma -\eta}}
<\infty, \quad\text{ if }\, \sigma >\eta,
\end{equation}
and interpret it as a control from above on the growth of the
weighted volume of geodesic balls with respect to Riemannian
distance function. This is a mild requirement, which is implied,
via a version of the Bishop--Gromov volume comparison theorem for
weighted manifolds, by a lower bound on the modified Bakry--Emery
Ricci curvature in the radial direction.
Indeed, as we shall see in Section~2 below,  the latter yields an
 upper estimate on $L_D  r$ which in turn gives the volume comparison estimate.
In fact, we shall prove there that an $L^p$-condition on the modified
Bakry--Emery Ricci curvature implies
a control from above on the weighted volume of geodesic balls.

On the contrary, as in the classical case of Riemannian geometry,
volume growth restrictions do not provide in general a control on $L_D
r$. This in turn prevents the possibility of constructing radial
super-solutions of (the equation corresponding to)
(\ref{main_ineq}), that could be used, as in the proof of
Theorem~A, as suitable barriers to study the existence problem via
comparison techniques. This technical difficulty forces us to devise
a new approach in the proof of Theorem~B, based on a generalization of the weak
maximum principle introduced by the authors in
\cite{RigoliSalvatoriVignati-Revista}, \cite{PigolaRigoliSetti-Memoirs}
(see Section~5).

In Section 6
we implement our techniques to analyze differential inequalities of the type
\begin{equation}
\label{ineq_minus}
L_{D, \vp} u \geq b(x) f(u) \ell(\modnabla u) - g(u) h(\modnabla u),
\end{equation}
where $g$ and $h$ are continuous functions. Our first task is to find
an appropriate form of the Keller--Osserman condition. To this end,
we let
\begin{equation}
\label{rho_def}
\tag{$\rho$}
\rho\in C^0(\R^+_0), \quad \rho(t)\geq 0 \,\text{ on }\, \R^+_0,
\end{equation}
and define the function $\hat{F}(t)= \hat F_{\rho, \omega}$
depending on the real parameter $\omega$ by the formula
\begin{equation}
\label{F_rho_omega}
\hat F_{\rho,\omega} (t) = \int_0^t f(s)
e^{(2-\omega) \int_0^s \rho(z) dz} ds.
\end{equation}
Note that $\hat F$ is well defined because of our assumptions. We
assume that $t\vp'/\ell\in L^1(0^+)\setminus L^1(+\infty)$, define
$K$ as in (\ref{K_def})
and let $K^{-1}:\R^+_0\to \R^+_0$ be its inverse. The new version of the
Keller--Osserman condition that we shall consider is
\begin{equation}
\label{rhoKO}
\tag{$\rho$KO}
{\ds\frac{e^{\int_0^t \rho(z) dz}}{K^{-1}\bigl(\hat F(t)\bigr)}}\in
L^ 1(+\infty).
\end{equation}

Of course, when $\rho\equiv 0$ we recover condition (\ref{KO})
introduced above. As we shall see in Section~5, the two
conditions are in fact equivalent if $\rho\in L^1$ under some mild
additional conditions.

We prove

\begin{Theorem C}
\label{Theorem C}
Let $(M,\langle\,,\rangle)$ be a complete manifold satisfying
\begin{equation}
\tag{\ref{genRicci_lower_bound}}
\mathrm{Ricc}_{n,m} (L_D) \geq H^2(1+r^2)^{\beta/2},
\end{equation}
for some $n>m$, $H>0$ and $\beta\geq -2$.
Assume that  ($F_1$), ($L_1$), ($L_2$), ($\vp\ell$), ($\theta$), ($b_1$) and
($\ell$)
hold with $\mu\geq 0$, $\theta \leq 1$ and
\begin{equation}
\label{thetabetamu'}\tag{$\theta\beta\mu$'}
\begin{cases}
\theta < 1-\beta/2 - \mu, &\text{ if } \, \theta\leq 1, \, \mu>0 \\
\theta = 1-\beta/2 -
\mu, &\text{ if }\,  \theta< 1, \, \mu>0\\
\theta< 1-\beta/2, &\text{ if } \, \theta \leq 1, \, \mu=0.
\end{cases}
\end{equation}

Suppose also that
\begin{itemize}
\item[($h$)] $h\in C^0(\R^+_0), \quad 0\leq h(t)\leq Ct^2 \vp'(t)$ on
$\R^+_0$, for some $C>0$,
\item[($g$)] $g\in  C^0(\R^+_0), \quad g(t)\leq C\rho(t)$
on $\R^+_0$, for some $C>0$,
\end{itemize}
and $\rho$ satisfying (\ref{rho_def}). If (\ref{rhoKO}) holds with $\omega=\theta$  in
the definition of $\hat F$, then any  entire classical
weak solution $u$ of the differential inequality (\ref{ineq_minus})
either non-positive or constant. Moreover, if $u\geq 0$ and $\ell(0)>0$ then $u\equiv 0.$
\end{Theorem C}
As already observed, ($\vp\ell$) is not satisfied by the mean
curvature operator; however, a version of Theorem~C
can be given to handle this case, see Section~6 below.

As mentioned earlier, is some circumstances (\ref{rhoKO}) is
equivalent to (\ref{KO}). This is the case, for instance, in the
next

\begin{Corollary C1}
\label{Corollary C1} Let $(M,\langle\,,\rangle)$ be as in Theorem~C.
Assume that ($g$), ($F_1$), ($L_2$), ($\vp\ell$), ($L_1$),
($\theta$)  and (\ref{thetabetamu'}) hold. Suppose also that
\begin{equation*}
g_+(t)=\max \{0,g(t)\} \in L^1(+\infty).
\end{equation*}
If (\ref{KO}) holds, then any  entire classical weak
solution $u$ of
\begin{equation*}
\Delta_p u\geq b(x) f(u) \ell(\modnabla u ) - g(u) \modnabla u^p
\end{equation*}
is either non-positive or constant. Moreover, if $u\geq 0$ and $\ell(0)>0$
then $u\equiv 0.$.
\end{Corollary C1}

We conclude this introduction by observing that in the literature
have recently appeared other methods to obtain Liouville-type
results for differential inequalities such as (\ref{main_ineq}) or
(\ref{ineq_minus}). Among them we mention the important technique
developed by E.~Mitidieri and S.I.~Pohozaev, see, e.g.,
\cite{MitidieriPohozaev-Steklov}, which proves to be very effective
when the ambient space is $\R^m$. Their method, which involves the
use of cut-off functions in a non-local way,  may be adapted to a
curved ambient space, but is not suitable to deal with situations
where the volume of balls grows superpolynomially.

The paper is organized as follows:
\begin{itemize}
\item[1] Introduction.
\item[2] Geometric comparison
results.
\item[3] Proof of Theorem~A and
related results.
\item[4] A second version of Theorem~A.
\item[5] The weak maximum principle and
non existence of solutions with controlled growth.
\item[6] Proof of Theorem~C.
\end{itemize}

In the sequel $C$ will always denote a positive
constant which may vary from line to line.

\section{Comparison results}
\label{section_comparison}
In this section we consider the diffusion operator
\begin{equation}
\label{diffusion}
L_Du = \frac{1}{D} \mathrm{div} (D\nabla u) \qquad D \in
C^2(M)\quad , \quad D>0.
\end{equation}
and denote   by $r(x)$ the distance from a fixed origin $o$ in an
$m$-dimensional complete Riemannian manifold $(M,\langle\,
,\rangle)$. The Riemannian metric and the weight $D$ give rise to a
metric measure space, with measure $D\,dV$, $dV$ denoting the usual
Riemannian volume element. For ease of  notation in the sequel we
will drop the index $D$ and write $L_D=L$.

The purpose of this section is to collect the estimates for
$L r$ and for the weighted volume of Riemannian balls, that will be
used in the sequel. The estimates are derived assuming an upper bound for a family of
modified Ricci tensors, which account for the mutual interactions of
the geometry and the weight function.

Although most of the material is available in the literature (see,
e.g. D. Bakry and P. Emery \cite{Bakry-Emery_DiffusionsHypercontractives},
Bakry \cite{Bakry-hypercontractivite}, A.G. Setti \cite{Setti-Canadian},
Z. Qian \cite{Qian-EstimatesWeightedVolume}, Bakry and
Qian \cite{BakryQian-VolumeComparisonTheorems}, J. Lott \cite{Lott-Commentarii},
X.-D. Li \cite{LiX-D-JMathPuresAppl}), we are going to present a quick
derivation of the estimates for completeness and the convenience of
the reader.

We note that our method is somewhat different from that of most of
the above authors. In addition we will be able to derive weighted volume estimates
under integral type conditions on the modified Bakry--Emery Ricci curvature,
which extend to this setting results of S. Gallot \cite{Gallot-Asterisque},  P. Petersen
and G. Wei \cite{PetersenWei-GAFA}, and S. Pigola, M. Rigoli and Setti
\cite{PigolaRigoliSetti-TheBook}.

For $n>m$ we let $\mathrm{Ricc}(L)$ and $\mathrm{Ricc}_{n,m}(L)$ denote
the Bakry-Emery and the modified Bakry-Emery Ricci tensors defined in
(\ref{modified_Bakry-Emery_Ricci}).
%

The starting point of our considerations is the following version
of the Bochner--Weitzenb\"ock formula for the diffusion operator
$L$.

\begin{lemma}
\label{Bochner-Weitzenbock_lemma}
Let $u\in C^3(M)$, then
\begin{equation}
\label{Bochner-Weitzenbock_formula}
\frac12 L\bigl(\modnabla u^2\bigr) = |\mathrm{Hess} u|^2 + \langle \nabla L u,
\nabla u\rangle + \mathrm{Ricc}(L)(\nabla u, \nabla u).
\end{equation}
\end{lemma}
\begin{proof}
It follows from the definition of $L$ and the usual
Bochner--Weitzenb\"ock formula that
\begin{equation*}
\begin{split}
L(\modnabla u^2) &= \Delta(\modnabla u^2) + D^{-1} \langle \nabla D,
\nabla \modnabla u^2\rangle \\
&= 2 |\mathrm{Hess} u|^2 +2\langle\nabla \Delta u, \nabla u \rangle + 2
\mathrm{Ricc}(\nabla u ,\nabla u ) + D^{-1} \langle \nabla D,
\nabla \modnabla u^2\rangle.
\end{split}
\end{equation*}
Now computations show that
\begin{equation*}
D^{-1} \langle \nabla D,
\nabla \modnabla u^2\rangle = 2D^{-1} \mathrm{Hess} u( \nabla u,
\nabla D)
\end{equation*}
and
\begin{equation*}
\begin{split}
\langle\nabla \Delta u, \nabla u \rangle &= \langle\nabla\bigl(L
u - D^{-1} \langle \nabla D, \nabla u\rangle\bigr), \nabla u\rangle \\
&= \langle\nabla (L u), \nabla u\rangle  + D^{-2} \langle \nabla u,
\nabla D\rangle^2 \\
&\hskip 0.5cm - D^{-1} \mathrm{Hess} u(\nabla u, \nabla D) -
D^{-1} \mathrm{Hess } D (\nabla u, \nabla u),
\end{split}
\end{equation*}
so that substituting yields the required conclusion.
\end{proof}

\begin{lemma}
\label{Lr_formula_lemma}
Let $(M, \langle\, ,  \rangle)$ be  a complete Riemannian manifold of
dimension $m$. Let $r(x)$ be the Riemannian distance function
from a fixed reference point $o$, and denote with $\mathrm{cut} (o)$ the
cut locus of $o$. Then for every $n>m$ and  $x\not\in \{o\}\cup \mathrm{cut} (o)$
\begin{equation}
\label{Lr_diff_ineq}
\frac 1{n-1} (Lr)^2 + \langle \nabla Lr, \nabla r\rangle +
\mathrm{Ricc} _{n,m}(\nabla r, \nabla r)\leq 0.
\end{equation}
\end{lemma}
\begin{proof}
We use $u=r(x)$ in the generalized Bochner--Weitzenb\"ock formula
(\ref{Bochner-Weitzenbock_formula}). Since
$\mathrm{Hess} r (\nabla r , X ) = 0$ for every vector field $X$, by
taking an orthonormal frame in the orthogonal complement of $\nabla
r$, and using the Cauchy--Schwarz inequality we see that
\begin{equation*}
|\mathrm{Hess}\, r|^2 \geq \frac 1{m-1} (\Delta r) ^2
\end{equation*}
Using the elementary inequality
\begin{equation*}
(a-b)^2\geq \frac1{1+\epsilon} a^2 - \frac 1 \epsilon b^2, \quad
a,b\in \R ,\,\epsilon  >0,
\end{equation*}
we estimate
\begin{equation*}
(\Delta u)^2 = (L u -D^{-1}\langle \nabla D, \nabla u\rangle)^2
\geq \frac 1{1+\epsilon} (Lu)^2 -\frac 1 \epsilon D^{-2}\langle
\nabla D, \nabla u\rangle^2.
\end{equation*}
Now, the required conclusion follows   substituting into (\ref{Bochner-Weitzenbock_formula}),
using $\modnabla r =1$,  choosing $\epsilon$ in such a way that $(1+\epsilon)(m-1) =n-1$,
and recalling the definition of $\mathrm{Ricc}_{n,m}$.
\end{proof}

We are now ready to prove the weighted Laplacian comparison theorem.
Versions of this results have been obtained by Setti, \cite{Setti-Canadian}, for
the case where $n=m+1$ and later by Qian  (\cite{Qian-EstimatesWeightedVolume})
in the general case where $n>m$ (see also \cite{BakryQian-VolumeComparisonTheorems}
which deals with the case where the drift term is not even assumed to be a gradient).
We present a proof modeled on the proof of the Laplacian Comparison Theorem
described in \cite{PigolaRigoliSetti-Memoirs}.

\begin{proposition}
\label{weighted_laplacian_comparison}
Let $(M, \langle\, ,  \rangle)$ be  a complete Riemannian manifold of
dimension $m$. Let $r(x)$ be the Riemannian distance function
from a fixed reference point $o$, and denote with $\mathrm{cut} (o)$ the
cut locus of $o$.
Assume that
\begin{equation}
\label{modifiedRicci_lower_bound}
\mathrm{Ricc}_{n,m}(\nabla r,\nabla r) \geq -(n-1) G(r)
\end{equation}
for some $G\in C^0([0,+\infty)$,  let $h\in C^2([0,+\infty)$
be  a solution of the problem
\begin{equation}
\label{h_ineq}
\begin{cases}
h'' - G h \geq 0\\
h(0)=0,\,\, h'(0)=1,
\end{cases}
\end{equation}
and let  $(0,R)$, $R\leq +\infty$, be  the maximal interval where $h(r)>0$. Then for
every $x\in M$ we have $r(x)\leq R$, and the inequality
\begin{equation}
\label{estimate_Lr}
L r(x) \leq (n-1)\frac {h'(r(x))}{h(r(x))}
\end{equation}
holds pointwise in $M\setminus (\mathrm{cut}(o)\cup \{o\})$ and
weakly on $M$.
\end{proposition}
\begin{proof}

Next let $x\in M\setminus (\mathrm{cut}(o)\cup \{o\})$,  let
$\gamma:[0, r(x)]\to M$ be the unique minimizing geodesic parametrized by arc length joining
$o$ to $x$, and set $\psi (s)= (L r)\circ \gamma (s)$. It follows from
(\ref{Lr_diff_ineq}) and $\dot\gamma = \nabla r$ that
\begin{equation}
\label{psi_ineq}
\begin{split}
\frac d{ds} (L r \circ\gamma)(s)
&= \langle \nabla L r ,\nabla r\rangle \circ \gamma \\
&\leq - \frac 1 {n-1} (L r \circ\gamma)(s)^2 + (n-1) G(s)
\end{split}
\end{equation}
on $(0, r(x))$. Moreover,
\begin{equation}
\label{psi_asympt}
(L r \circ\gamma)(s) = \frac{m-1} s + O(1) \quad \text{as }\, s\to 0+,
\end{equation}
which follows from the fact that
\begin{equation*}
(L r \circ\gamma)(s) = \bigl(\Delta r + D^{-1}\langle \nabla D, \nabla r\rangle\bigr)
\circ\gamma(s)
\end{equation*}
and the second summand is bounded as $s\to 0+$, while, by standard
estimates,
\begin{equation*}
\Delta r(x) = \frac{m-1} {r(x)} + o(1).
\end{equation*}
Because of (\ref{psi_asympt}), we may set
\begin{equation}
\label{g_def}
g(s) = s^{\frac{m-1}{n-1}} \exp\left(
\int_0^s \bigl[\frac{(L r \circ\gamma)(t)}{n-1} - \frac{m-1}{n-1} \frac 1 t
\bigr] dt\right),
\end{equation}
so that $g$ is defined in $[0, r(x)],$ $g(s)>0$ in $(0, r(x))$,
and it satisfies
\begin{equation}
\label{g_prop}
(n-1)\frac {g'}{g} = L r \circ\gamma,\quad  g(0) = 0,\quad
g(s) = s^{\frac{m-1}{n-1}} \bigl(1 + o(1)\bigr) \,\, \text{as }
s\to 0+.
\end{equation}
It follows from this and (\ref{psi_ineq}) that $g$ satisfies the problem
\begin{equation}
\label{g_ineq}
\begin{cases}
g'' \leq G g &\\
g(0) = 0, \quad g'(s)= s^{\frac{m-n}{n-1}} \bigl(1 + o(1)\bigr)
\, \text{ as } s\to 0^+.
\end{cases}
\end{equation}
Recalling that, by assumption $h$ satisfies (\ref{h_ineq}), we now proceed as
in the standard Sturm comparison theorems, and consider the function
\begin{equation*}
z(s) = h'(s)g(s)- h(s) g'(s).
\end{equation*}
Then
\begin{equation*}
z'(s) = gh  \bigl(\frac {h''} h - \frac{g''}g\bigr) \geq 0
\end{equation*}
in the interval $(0, \tau)$, $\tau=\min\{r(x), R\}$, where $g$ is defined and $h$
is positive. Also, it follows from the asymptotic
behavior of $g$ and $h$ that
\begin{equation*}
h'(s) g(s) \asymp s^{\frac{m-1}{n-1}},\quad h(s)g'(s)\asymp
\frac{m-1}{n-1} s^{\frac{m-1}{n-1}}
\end{equation*}
so that
\begin{equation*}
z(s)\to 0^+ \,\text{ as } s\to 0.
\end{equation*}
We conclude that $z(s)\geq 0$ and therefore
\begin{equation*}
\frac {g'(s)}{g(s)}\leq \frac {h'(s)}{h(s)}
\end{equation*}
in the interval $(0,\tau)$.

Integrating between $\epsilon$ and $s$, $0<\epsilon<s <\tau$, yields
\begin{equation*}
g(s)\leq \frac{g(\epsilon)}{h(\epsilon)} h(s),
\end{equation*}
showing that $h$ must be positive in $(0,\tau)$, and therefore $r(x)\leq R.$
Since this holds for every $x\in M$ we deduce that if $R<+\infty$ then $M$ is compact and
$diam(M)\leq 2R$. Moreover, in $(0,r(x))$ we have
\begin{equation*}
(L_ r) (\gamma(r(x))) = (n-1) \frac{g'}g (r(x)) \leq (n-1) \frac{h'}h
(r(x)).
\end{equation*}
This shows that the inequality (\ref{estimate_Lr}) holds pointwise in
$M\setminus (\mathrm{cut}(o)\cup \{o\})$. The weak inequality now
follows from standard arguments (see, e.g., \cite{PigolaRigoliSetti-Memoirs},
Lemma~2.2, \cite{PigolaRigoliSetti-TheBook}, Lemma 2.5).
\end{proof}

As in the standard Riemannian case, the estimate for $Lr$ allows to
obtain weighted volume comparison
estimates (see, \cite{Setti-Canadian}, \cite{Qian-EstimatesWeightedVolume},
\cite{BakryQian-VolumeComparisonTheorems}, \cite{LiX-D-JMathPuresAppl}).
\begin{theorem}
\label{theorem_volume_comparison}
Let $\left(  M,\left\langle\, ,\right\rangle \right)  $ be as in the
previous Proposition, and assume that the modified Bakry-Emery Ricci
tensor $\mathrm{Ricc}_{n,m}$ satisfies
(\ref{modifiedRicci_lower_bound})
for some $G\in C^0([0,+\infty)$.  Let $h\in C^2([0,+\infty)$
be  a solution of the problem (\ref{h_ineq}), and
let $(0,R)$ be the maximal interval where $h$ is positive.
Then, the functions
\begin{equation}
\label{volume_comparison1}%
r\mapsto \frac{\mathrm{vol_D}\partial B_{r}(o)}{h(r)^{n-1}}
\end{equation}
and
\begin{equation}
\label{volume_comparison2}
r \mapsto \frac{\mathrm{vol}_D B_{r}(  o)}{\int_0^r  h(t)^{n-1}dt},
\end{equation}
are non-increasing a.e, respectively non-increasing, in $(0,R)$.
In particular, for every $0< r_o <R$,  there exists a constant $C$ depending on
$D$ and on the geometry of $M$ in $B_{r_o}(o)$ such that
\begin{equation}
\label{volume_comparison3}
\mathrm{vol}_D(B_{r}(o))\leq  C
\begin{cases}
r^m & \text{ if }\, 0\leq r\leq r_o\\
\int_0^{r} h(t)^{n-1} dt &\text{ if }\, r_o\leq r.
\end{cases}
\end{equation}
\end{theorem}
\begin{proof}
By Lemma \ref{weighted_laplacian_comparison}, inequality  (\ref{estimate_Lr})
holds weakly on $M$, so for every $0\leq\varphi\in
Lip_{c}\left(  M\right)  $, we have
\begin{equation}
-\int\left\langle \nabla r,\nabla\varphi\right\rangle D(x) dV
\leq
\left(  n-1\right)
\int \varphi\frac{h^{\prime}\left(  r\left(  x\right)  \right)  }{h\left(  r\left(
x\right)  \right)  } D(x) dV. \label{volume_comparison4}%
\end{equation}
For any $\varepsilon>0,$ consider the radial cut-off function%
\begin{equation}
\label{volcomp4a'}
\varphi_{\varepsilon}\left(  x\right)  =
\rho_{\varepsilon}\left(  r\left(x\right)  \right)
h(r(x))  ^{-n+1}%
\end{equation}
where $\rho_{\varepsilon}$ is the piecewise linear function%
\begin{equation}
\label{volcomp4a''}
\rho_{\varepsilon}\left(  t\right)  =\left\{
\begin{array}
[c]{ll}
0 &\text{if } t\in [0,r)\\
\frac{t-r}{\varepsilon} &\text{if } t\in [r, r+\varepsilon)\\
1 & \text{if }t\in [r+\varepsilon ,R-\varepsilon)\\
\frac{R-t}{\varepsilon} & \text{if }t\in [R-\varepsilon,R)\\
0 & \text{if }t\in [R,\infty).
\end{array}
\right.
\end{equation}
Note that
\[
\nabla\varphi_{\varepsilon}=\left\{
-\frac{\chi_{R-\varepsilon, R}}{\varepsilon}
+\frac{\chi_{r, r+\varepsilon}}{\varepsilon}
-\left(  n-1\right)
\frac{h^{\prime}(r(x))}{ h(r(x))}\rho_{\varepsilon}
\right\}
h\left( r(x)\right)  ^{-n+ 1}
\nabla r,
\]
for a.e. $x\in M$, where $\chi_{s,t}$ is the characteristic function
of the annulus $B_{t}\left(  o\right)  \setminus B_{s}\left(  o\right)  .$
Therefore, using $\varphi_{\varepsilon}$ into (\ref{volume_comparison4}) and
simplifying, we get%
\begin{equation*}
\frac{1}{\varepsilon} \int_{B_{R}\left(  o\right)\setminus
B_{R-\varepsilon}\left(o\right)  } h\left(  r(x)\right)  ^{-n+1}
\leq \frac 1\varepsilon \int_{B_{r+\varepsilon}\left(
o\right)\setminus B_{r}\left(o\right)  } h\left(  r(x)\right)
^{-n+1}.
\end{equation*}
Using the co-area formula we deduce that
\[
\frac{1}{\varepsilon}\int_{R-\varepsilon}^{R}
\mathrm{vol}\partial B_{t}\left(o\right)
h\left(  t\right)  ^{-n+1}
\leq
\frac{1}{\varepsilon}\int_{r}^{r+\varepsilon}
\mathrm{vol}\partial B_{t}\left(o\right)
h\left(  t\right)  ^{-n+1}
\]
and, letting $\varepsilon\searrow0$,
\begin{equation}
\frac{\mathrm{vol}_D\partial B_{R}\left(  o\right)} {h\left(R\right)  ^{m-1}}
\leq
\frac{\mathrm{vol}_D\partial B_{r}\left(  o\right)} {h\left(r\right)  ^{m-1}}
\label{volcomp4c}%
\end{equation}
for a.e. $0<r<R$.
The second statement follows from the first and the co-area formula, since,
as noted by M.~Gromov (see, \cite{CheegerGromovTaylor-JDG}),
for general real valued functions $f\left(  t\right)  \geq0$, $g\left(  t\right)  >0$,
\[
\text{if } \, \,
t\to \frac{f\left(  t\right)  }{g\left(  t\right)  }\,\,
\text{ is decreasing, then  }\,\,
t\to \frac{\int_{0}^{t}f}{\int_{0}^{t}g}\, \, \text{is decreasing}.
\]
\end{proof}
\goodbreak

We  next consider the situation where the modified Bakry-Emery
Ricci curvature satisfies some $L^p$-integrability conditions and
extends results obtained in \cite{PigolaRigoliSetti-TheBook} for the
Riemannian volume which in turn slightly generalize previous results by
P. Petersen and G. Wei, \cite{PetersenWei-GAFA}
(see also \cite{Gallot-Asterisque} and \cite{LiYau-Compositio}).

Since we will be interested in the case the underlying manifold is non-compact,
we assume that $G$ is a non-negative, continuous  function on
$[0,+\infty)$ and that
$h\left(  t\right) \in C^{2}\left([0,+\infty)\right)  $ is the
solution of the problem%
\begin{equation*}
\left\{
\begin{array}
[c]{l}%
h^{\prime\prime}\left(  t\right)  -G\left(  t\right)  h\left(  t\right)
=0\\
h\left(  0\right)  =0 \,\, h^{\prime}\left(  0\right)  =1.
\end{array}
\right.
\end{equation*}
The assumption that $G\geq 0$ implies that $h'\geq 1$ on $[0,+\infty)$
and therefore $h>0$ on $(0,+\infty).$
For ease of notation, in the course of the arguments that follow we
set
\begin{equation}
\label{Petersen eq3}
A_{G,n}  (r)= h(r)^{n-1}\quad \text{and} \quad V_{G,n}(r) = \int_0^r h(t)^{n-1} dt
\end{equation}
so that $A_{G,n}(r)$ and $V_{G,n}(r)$ are multiples of the measures of the sphere and of the ball
of radius $r$ centered at the pole in the $n$-dimensional model manifold $M_G$ with radial
Ricci curvature equal to $-(n-1)G$.


Using an exhaustion of  $E_o=M\setminus \mathrm{cut}(o)$ by means
of starlike domains one shows (see, e.g., \cite{PigolaRigoliSetti-TheBook},  p. 35)
that for every non-negative test
function $\vp\in Lip_c(M)$,
\begin{equation}
\label{Lr_weak}
-\int_M \langle \nabla r, \nabla \vp\rangle D dV \leq \int_{E_o} \vp
L r D dV.
\end{equation}
We outline the argument for the convenience of the reader. Let
$\Omega_n$ be such an exhaustion of $E_o$, so that, if $\nu_n$
denotes the outward unit normal to $\partial \Omega_n$, then
$\langle \nu_n, \nabla r\rangle\geq 0.$ Integrating by parts shows
that
\begin{equation*}
\begin{array}{l}
\displaystyle -\int_M \langle \nabla r, \nabla \vp\rangle D dV
= -\lim_n \int_{\Omega_n} \langle \nabla r, \nabla \vp\rangle D
dV \\[0.4cm]
\displaystyle = \lim_n\Big\{ \int_{\Omega_n}\vp [\Delta r +\frac 1D\langle \nabla D, \nabla r\rangle ] DdV
 -\int_{\partial \Omega_n} \!\!\!\vp \langle \nabla r, \nu_n\rangle D d\sigma \Big\} \\[0.4cm]
\displaystyle \leq \lim_n \int_{\Omega_n}\vp L_D r DdV
=\int_{E_o}\vp  L_D r DdV,
\end{array}
\end{equation*}
where the inequality follows from $\langle\nabla r, \nu_n\rangle
\geq 0$, and the limit on the last line exists because, by
Proposition \ref{weighted_laplacian_comparison}, $L r$ is bounded
above by some positive integrable function $g$ on the relatively compact set
$E_o\cap \text{supp} \vp$ (namely, if
$\mathrm{Ricc}_{m,n} \ge -(n-1)H^2$ on $E_o\cap \text{supp} \vp$
for some $H>0$, we can choose
$g= H\coth(Hr)$).\\
Applying the
above inequality to the test function
$$
\vp_\epsilon(x)=\rho_\epsilon(r(x)) h(r(x))^{-n+1},
$$
already considered in (\ref{volcomp4a'}), arguing as in the proof of
Theorem~\ref{theorem_volume_comparison},
and using the fact that $A_{G,n}(r) = h(r)^{n-1}$ is non-decreasing,
we deduce that for a.e. $0<r<R$
\begin{equation}
\label{Petersen eq0}
\begin{split}
\frac{\volD \partial B_R}{A_{G,n}(R)} - \frac{\volD \partial B_r}{A_{G,n}(r)}
\leq \frac1{A_{G,n}(r)}  \int_{B_R\setminus B_r} \psi D dV,
\end{split}
\end{equation}
where we have set
\begin{equation}
\label{Petersen eq1}
\psi(x)
=
\begin{cases}
\max\{ 0, L r(x) - (n-1) \frac{h'(r(x))}{h(r(x))} \} &\text{if } \,
x\in E_o\\
0&\text{if } \, x\in \mathrm{cut}(o).
\end{cases}
\end{equation}
Note by virtue of the asymptotic behavior of $Lr$ and $h'/h$ as $r(x)\to
0$, $\psi$ vanishes in a neighborhood of $o$. Moreover, if
$\mathrm{Ric}_{n,m} (\nabla r, \nabla r)\geq  -(n-1) G(r(x))$, then,
by the weighted Laplacian comparison theorem, $\psi(x)\equiv 0,$ and
we recover the fact that the function
\begin{equation}
r\to\frac{\volD \partial B_r}{A_{G,n}(r)}
\end{equation}
is non-increasing for a.e. $r$.
%

Using the co-area formula, inserting (\ref{Petersen eq0}), and
applying H\"older inequality with exponents $2p$ and $2p/(2p-1)$
to the right hand side of the resulting inequality
we conclude that
\begin{equation}
\label{Petersen eq4}
\begin{split}
\frac{d}{dR} \left(\frac{\vol_D B_R(o)}{V_{G,n} (R)}\right)
&= \frac{V_{G,n}(R)\volD \partial B_R  - A_{G,n}(R)\volD B_R } {V_{G,n}(R)^2}
\\ & =V_G(R)^{-2}\int_0^R
\bigl( A_{G,n}(r)\vol \partial B_R  - A_{G,n}(R)\volD \partial B_r \bigl)\,
dr
\\&
\leq\frac{R A_{G,n}(R) } {V_{G,n}(R)^{1+1/2p}}
\Bigl(\frac{\volD B_R}{V_{G,n}(R)}\Bigr)^{1-1/2p} \Bigl(\int_{B_R} \psi
^{2p} D dV\Bigr)^{1/2p}
\end{split}
\end{equation}
Now we define
\begin{equation}
\label{Petersen eq2}
\begin{split}
\rho(x)
&=  -\min
\{
0,
 \mathrm{Ric}_{n,m}(\nabla r, \nabla r) + (n-1)G (r(x))\}\\
 &=  \bigl[\mathrm{Ric}_{n,m}(\nabla r, \nabla r) + (n-1)G
(r(x))\bigr]_-.
\end{split}
\end{equation}
We will need to estimate the integral on the right hand side of
(\ref{Petersen eq4}) in terms of $\rho$. This is achieved in the
following lemma, which is a minor modification of
\cite{PetersenWei-GAFA}, Lemma 2.2,
and \cite{PigolaRigoliSetti-TheBook}, Lemma~2.19.
%
%
%

\begin{lemma}
\label{Petersen lemma2} For every $p>n/2$ there exists a constant
$C=C(n,p)$ such that for every $R$
\begin{equation*}
\int_{B_R} \psi^{2p} D dV  \leq C \int_{B_R} \rho ^p D dV.
\end{equation*}
with $\rho(x)$ defined in (\ref{Petersen eq2}).
\end{lemma}

\begin{proof}
Integrating in polar geodesic coordinates we have
\begin{equation*}
\int_{B_R} f DdV = \int_{S^{m-1} }d\theta \int_0^{\min\{R, c(\theta)\}}
f(t\theta){(D\omega)(t\theta)}   dt
\end{equation*}
where $ {\omega}$ is the volume density
with respect to Lebesgue measure $dtd\theta$, and $c(\theta)$ is the distance from $o$
to the cut locus along the ray $t\to t\theta$.
It follows that
it suffices to prove that for every $\theta\in S^{m-1}$
\begin{equation}
\label{Petersen eq5}
\int_0^{{\min\{R, c(\theta)\}}} \psi^{2p}(t\theta) {(D\omega)(t\theta)}  dt \leq C
\int_0^{{\min\{R, c(\theta)\}}} \rho^{p}(t\theta) {(D\omega)(t\theta)}
dt.
\end{equation}
An easy computation which uses (\ref{psi_ineq}) yields
\begin{equation*}
\frac{\partial}{\partial t} \{ L r - (n-1)\frac{h'}{h}\}
\leq
-\frac {(L r)^2}{n-1}  - Ric_{n,m}( \nabla r, \nabla r) - (n-1)
\Bigl\{  \frac {h''}h -
\Bigl(\frac{h'}{h}\Bigr)^2\Bigr\}
\end{equation*}
Thus, recalling the definitions of $\psi$ and $\rho$, we deduce that
the locally Lipschitz function $\psi$ satisfies the
differential inequality
\begin{equation*}
\psi' + \frac{\psi^2}{n-1} +2\frac {h'}h \psi \leq \rho,
\end{equation*}
on the set where $\rho>0$ and a.e. on $(0,+\infty)$.
Multiplying through by $\psi^{2p-2} D\omega,$ and integrating
we obtain
\begin{equation}
\label{Petersen eq6}
\int_0^{r}\bigl(\psi' \psi^{2p-2}  + \frac 1 {n-1}
\psi^{2p}  + 2 \frac{h'}h \psi^{2p-1}\bigr) D\omega
\leq \int_0^{r} \rho \psi^{2p-2} D \omega.
\end{equation}
On the other hand, integrating by parts, and recalling that
\begin{equation*}
(D\omega)^{-1} \partial (D\omega)/\partial t = L r\leq \psi + (n-1) \frac {h'}h
\end{equation*}
and that $\psi(t\theta)= 0$ if $t\geq c(\theta)$, yield
\begin{equation*}
\begin{split}
\int_0^r \psi' \psi^{2p-2} \omega &= \frac 1{2p-1} \psi(r)^{2p-1}
(D\omega)(r\theta)  - \frac 1 {2p-1} \int_0^r \psi^{2p-1} L r
D\omega \\
&\geq
- \frac 1 {2p-1} \int_0^r \psi^{2p-1}
\bigl(\psi + (n-1)\frac{h'}h \bigr)
D\omega.
\end{split}
\end{equation*}
Substituting this into (\ref{Petersen eq6}), and using H\"older inequality
we obtain
\begin{equation*}
\begin{split}
\Bigl(\frac 1{n-1}- \frac 1 {2p-1}\Bigr) \int_0^r \psi^{2p} D\omega
&+ \Bigl(2 - \frac {n-1} {2p-1}\Bigr) \int_0^r \psi^{2p-1}\frac {h'}h
D\omega \\
&\leq \int_0^r \rho\psi^{2p-2} D\omega \\
&\leq\Bigl( \int _0^r \rho^p
D\omega\Bigr)^{1/p}\Bigl(\int_0^r\psi^{2p}D\omega \Bigr)^{(p-1)/p},
\end{split}
\end{equation*}
and, since the coefficient of the first integral
on the left hand side is positive, by the assumption on $p$, while
the second summand is nonnegative, rearranging and simplifying we
conclude that (\ref{Petersen eq5}) holds with
\begin{equation*}
C(n,p)=\Bigl(\frac 1{n-1}- \frac 1 {2p-1}\Bigr)^{-p}.
\end{equation*}
\end{proof}

We are now ready to state the announced weighted volume comparison theorem
under  assumptions on the $L^p$ norm of the modified Bakry-Emery Ricci curvature.

\begin{theorem}
\label{Petersen theorem}
Keeping  the notation introduced above, let  $p>n/2$
and let
\begin{equation}
\label{Petersen eq8}
f(t) = \frac{C_{n,p}^{1/2p} t A_{G,n}(t)}{V_{G,n}(t)^{1+1/2p}}
\Bigl(\int_{B_t} \rho^{p} DdV \Bigr)^{1/2p}.
\end{equation}
where $C_{n,p}$ is the constant in Lemma~\ref{Petersen lemma2}.
Then  for every $0<r<R$,
\begin{equation}
\label{Petersen eq7}
\Bigr(\frac{\volD B_R(o)}{V_{G,n} (R)}\Bigr)^{1/2p} -\Bigr(\frac{\volD B_r(o)}{V_{G,n}
(r)}\Bigr)^{1/2p}
\leq   \frac1{2p} \int_r^R f(t) dt.
\end{equation}
Moreover
for every $r_o>0$ there exists a constant $C_{r_o}$  such that, for
every $R\geq r_o$
\begin{equation}
\label{Petersen eq7'}
\frac{\volD B_R(o)}{V_{G,n} (R)}
\leq   \Bigl(C_{r_o} + \frac1{2p} \int_{r_o}^R f(t) dt \Bigr)^{2p},
\end{equation}
and
\begin{equation}
\label{Petersen eq7''}
\begin{split}
\frac{\volD \partial B_R(o)}{A_{G, n} (R)}
& \leq  \Bigl(C_{r_o} + \frac1{2p}\int_{r_o}^R f(t) dt \Bigr)^{2p}
\\&
+ \frac{ R}{V_{G,n}(R)^{1/2p}}
\Bigl(\int_{B_R} \rho^{p} \Bigr)^{1/2p}\Bigl(C_{r_o}
+ \frac 1 {2p} \int_{r_o}^R f(t) dt \Bigr)^{2p-1}
\end{split}
\end{equation}
\end{theorem}

\begin{proof}
Set
\begin{equation*}
y(r) =\frac{\vol_D B_r(o)}{V_{G,n} (r)}.
\end{equation*}
According to (\ref{Petersen eq4})
 Lemma~\ref{Petersen lemma2} and (\ref{Petersen eq8}) we have
\begin{equation*}
\begin{cases}
y'(t)  \leq
 f(t) y(t)^{1-1/2p},&\\
 y(t)\sim c_m t^{m-n} \text{ as } t\to 0+,\quad  y(t)>0 \,\, \text{ if } t>0.
\end{cases}
\end{equation*}
whence, integrating between $r$ and $R$ we obtain
\begin{equation*}
y(R)^{1/2p} - y(r)^{1/2p}\leq \frac{1}{2p} \int_r^R f(t) \, dt,
\end{equation*}
that is, (\ref{Petersen eq7}), and (\ref{Petersen eq7'}) follows at
one with $C_{r_o}= \Bigr(\frac{\volD B_{r_o}(o)}{V_{G,n}
(r_o)}\Bigr)^{1/2p}$.
On the other hand, according to (\ref{Petersen eq4}) and
Lemma~\ref{Petersen lemma2},
\begin{equation*}
\frac{\volD \partial B_R} {A_{G,n}(R)}  \leq \frac{\volD B_R } {V_{G,n}(R)}
+
\frac{ R}{V_{G,n}(R)^{1/2p}}
\Bigl(\int_{B_t} \rho^{p} DdV \Bigr)^{1/2p}
\Bigl(\frac{\volD B_R } {V_{G,n}(R)}\Bigr)^{1-1/2p}
\end{equation*}
and the conclusion follows inserting (\ref{Petersen eq7'}).
\end{proof}

Keeping the  notation introduced above,
assume, for instance, that $G=B^2\geq 0$, so that
\begin{equation*}
A_{G,n}(t) =
\begin{cases}
t^{n-1} &\text{ if } B=0\\
(B^{-1}\sinh Bt)^{n-1} &\text{  if } B>0
\end{cases}
\end{equation*}
and suppose that
\begin{equation*}
\rho=[ \mathrm{Ric_{n,m}}+ (n-1)B^2]_- \in L^p(M, DdV),
\end{equation*}
for some $p>n/2$. Then, arguing as in the proof
of \cite{PigolaRigoliSetti-TheBook}  Corollary 2.21,
we  deduce that  for every $r_o$
sufficiently small there exist constants $C_1$ and $C_2$, depending  on
$r_o$, $B$ $m$   $p$ and on the $L^p(M, DdV)$-norm of $\rho$,
such that, for every $R\geq r_o$,
\begin{equation*}
\volD B_R \leq C_1
\begin{cases}
R^{2p} &\text{ if } B=0\\
e^{(n-1)BR} &\text{ if } B>0.
\end{cases}
\end{equation*}
and
\begin{equation*}
\volD \partial B_R \leq
C_2
\begin{cases}
R^{2p-1} &\text{ if } B=0\\
e^{(n-1)BR} &\text{ if } B>0.
\end{cases}
\end{equation*}

\section{Proof of Theorem~A and further results}
\label{section_proof_ThmA}
The aim of this section is to give a proof of a somewhat stronger
form of Theorem~A (see Theorem~\ref{thm3.21} below), together with a version
of the result valid when (\ref{KO}) fails.

The idea of proof of Theorem~A is to  construct a function $v(x)$ defined on
an annular region $B_{\bar R}\setminus B_{r_o}$, with $0<r_o<\bar R$ sufficiently large,
with the following properties: for fixed $r_o<r_1<\bar R$ and $0<\epsilon<\eta$
\begin{equation}
\label{eq3.1}
\begin{cases}
v(x)=\epsilon &\text{on } \bdr B_{r_o}\\
\epsilon \leq v(x)\leq \eta &\text{on } B_{r_1}\setminus B_{r_o}\\
v(x)\to +\infty &\text{as } r(x)\to +\infty,
\end{cases}
\end{equation}
and $v$ is a weak supersolution on $B_{\bar R}\setminus B_{r_o}$ of
\begin{equation}
\label{eq3.2}
L_{D,\vp} w =   b(x) f(w)\ell(|\nabla w|).
\end{equation}
This is achieved by taking $v$ of the form
\begin{equation}
\label{eq3.3}
v(x) = \alpha(r(x))
\end{equation}
where $\alpha$ is a suitable supersolution of the radialized
inequality (\ref{eq3.2}), whose construction depends in a crucial way on the
validity of the Keller--Osserman condition (\ref{KO}).

The conclusion is then reached comparing $v$ with the solution of
(\ref{main_ineq}). To this end, we will extend a comparison
technique first introduced in \cite{PigolaRigoliSetti-JFA2002}

Finally, in Theorem~\ref{thm3.35} below we will consider the case where the
Keller--Osserman condition fails, that is
\begin{equation}
\label{not_KO}
\frac 1 {K^{-1}\bigl(F(t)\bigr)} \not\in L^1(+\infty).
\end{equation}
Its proof is based on a modification of the previous arguments and
uses (\ref{not_KO}) in a way which is, in some sense, dual to the
use of (\ref{KO}) in the proof of Theorem~A.

We begin with the following simple

\begin{lemma}
\label{lemma_KOsigma}
Assume that $f$, $\ell$ and $\vp$ satisfy the assumptions ($F_1$),
($L_1$) and ($\vp\ell$)$_2$, and let $\sigma>0$. Then (\ref{KO}) holds
if and only if
\begin{equation}
\label{KOsigma}
\tag{KO$\sigma$}
\frac 1{K^{-1}(\sigma F(s))}\in L^1(+\infty).
\end{equation}
\end{lemma}

\begin{proof}
We consider first the case $0<\sigma\leq 1$.
Since $K^{-1}$ is non-decreasing,
\begin{equation*}
\int^{+\infty}
\frac {ds}{K^{-1}(F(s))} \leq\int^{+\infty}
\frac 1{K^{-1}(\sigma F(s))}.
\end{equation*}
On the other hand, if  $C\geq 1$ is such that $\sup_{s\leq t}f(s)\leq C
f(t)$, then, for every $0<\sigma\leq1$, $f(C\sigma^{-1} t) \geq C^{-1}
f(t)$ and
\begin{equation*}
F(\frac{Ct}\sigma ) =\int_0^{\frac{Ct }\sigma} f(z) dz = \frac C\sigma
\int_0^t f(\frac{C\xi}\sigma) d\xi \geq \frac 1 \sigma\int_0^tf(\xi)
d\xi = \frac 1\sigma F(t),
\end{equation*}
so,   using the monotonicity of $K^{-1}$, we obtain
\begin{equation*}
\int^{+\infty}
\frac {ds}{K^{-1}(\sigma F(s))} =\frac C\sigma\int^{+\infty}
\frac {dt} {K^{-1}(\sigma F(\frac{ Ct}\sigma ))}\leq  \frac C\sigma \int^{+\infty}
\frac {dt} {K^{-1}\bigl(F(t)\bigr)} ,
\end{equation*}
showing that (\ref{KO}) and (\ref{KOsigma}) are equivalent in the
case $\sigma \leq 1$.

Consider now the case $\sigma >1$, and set  $f_\sigma = \sigma f$, $F_\sigma = \sigma F$.
Since (\ref{KOsigma}) is precisely  (\ref{KO}) for $F_\sigma$,  and since $\sigma^{-1}\leq 1$,
by what we have just proved it is equivalent to
\begin{equation*}
\frac 1{K^{-1}\bigl(\sigma^{-1}F_\sigma(s)\bigr)}
=
\frac 1{K^{-1}\bigl(F(s)\bigr)}
\in L^1(+\infty),
\end{equation*}
as required.
\end{proof}

We note for future use that the conclusion of the lemma depends only on the
monotonicity of $K^{-1}$ and the $C$-monotonicity of $f$.

Before proceeding toward our main result we would like to explore the mutual connections
between ($\theta$) and ($\vp\ell$). To simplify the writing, with the
statement ``($\theta$)$_1$ holds" we will mean that the first half of
condition ($\theta$) is valid.

\begin{proposition}
\label{proposition_theta_vpell}
Assume that conditions ($\Phi_0$) and ($L_1$) hold. Then ($\theta$)$_1$
with $\theta<2$ implies ($\vp\ell$)$_2$, and ($\theta$)$_2$ with
$\theta<1$ implies ($\vp\ell$)$_1$. As a consequence, ($\theta$)
with $\theta<1$ implies ($\vp\ell$).
\end{proposition}

\begin{proof}
Assume ($\theta$)$_1$, that is, the function
$t\to\frac{\vp'(t)}{\ell(t)}t^\theta$ is $C$-increasing on $\R^+.$
By definition there exists $C\geq 1$ such that
\begin{equation*}
0<s^{\theta}\frac{\vp'(st)}{\ell(st)} \leq C
\frac{\vp'(t)}{\ell(t)} \quad \forall t\in \R^+\,\, s\in
(0,1],
\end{equation*}
or, equivalently,
\begin{equation}
\label{prop3.2_eq1}
s^{\theta}\frac{\vp'(st)}{\ell(st)} \geq C^{-1}
\frac{\vp'(t)}{\ell(t)} \quad \forall t\in \R^+\,\, s\in
[1, +\infty).
\end{equation}
Letting $t=1$, we deduce that if $\theta<2$ then $\frac{s\vp'(s)}{\ell(s)}\in
L^1(0+)\setminus L^1(+\infty)$, which is ($\vp\ell$)$_2$.

In an entirely similar way, if ($\theta$)$_2$ holds, that is,
\begin{equation*}
\frac{\vp(st)}{\ell(st)}(st)^{\theta -1}  \leq C
\frac{\vp(t)}{\ell(t)}(t)^{\theta -1} \quad \forall
t\in \R^+\,\, s\in (0,1],
\end{equation*}
and $\theta<1$, then $s^{\theta-1}\frac{\vp(s)}{\ell(s)}\in L^\infty((0,1))$, and
\begin{equation*}
\lim_{s\to 0+} \frac{\vp(s)}{\ell(s)} =0,
\end{equation*}
which implies ($\vp\ell$)$_1$.
\end{proof}

\begin{remark}\label{rmk_theta_vpell}
{\rm
Note that the above argument above also shows that if ($\theta$)$_2$ holds with $\theta<2$ then
$\frac{\vp(t)}{\ell(t)} \in L^1(0^+)\setminus L^1(+\infty)$.
}
\end{remark}

\begin{proposition}
\label{prop 3.2bis}
Assume that conditions ($\Phi_0$) and ($L_1$) hold, and let $F$ be a positive function
defined on $\R_0^+$. If ($\theta$)$_1$ holds with $\theta<2$, then there
exists a constant  $B\geq 1$ such that, for every $\sigma \leq 1$ we have
\begin{equation}
\label{eq_prop3.2bis_1}
\frac{\sigma^{1/(2-\theta)}}{K^{-1}(\sigma F(t))}\leq
\frac B{K^{-1}(F(t))}
\quad \text{ on } \R^+.
\end{equation}
\end{proposition}

\begin{proof}
Observe first of all that according to Proposition~\ref{proposition_theta_vpell},
($\theta$)$_{1}$ with $\theta<2$ implies ($\vp\ell$)$_2$, so that
$K^{-1}$ is well defined on $\R^+_0$.

Changing variables in the definition of $K$, and  using (\ref{prop3.2_eq1}) above,
for every $\lambda \geq 1$ and $t\in \R^+$, we have
\begin{equation*}
\begin{split}
K(\lambda t) &= \int_0^{\lambda t} s\frac{\vp'(s)}{\ell(s)}ds =
\lambda ^2 \int _0^t s\frac{\vp'(\lambda s)}{\ell(\lambda s)}ds \\
&\geq
C^{-1} \lambda ^{2-\theta} \int_0^{t} s\frac{\vp'(s)}{\ell(s)}ds
= C^{-1} \lambda ^{2-\theta} K(t),
\end{split}
\end{equation*}
where $C\geq 1$ is the constant in ($\theta$)$_1$.
Applying $K^{-1}$ to both sides of the above inequality, and setting
$t= K^{-1}(\sigma F(s))$ we deduce that
\begin{equation*}
\lambda K^{-1}(\sigma F(s)) \geq K^{-1}(\lambda^{2-\theta}\sigma
C^{-1} F(s)),
\end{equation*}
whence, setting $\lambda =(C/\sigma)^{1/(2-\theta)}\geq
1$, the required conclusion follows with $B=C^{1/(2-\theta)}$.
\end{proof}

\begin{remark}
\label{rmk prop 3.2bis}
{\rm
We note for future use that the estimate holds for any positive function $F$ on $\R^+$,
without any monotonicity property, and it depends only on the
the fact that the integrand $\psi(s) = s\vp'(s)/\ell(s)$
in the definition of $K$ satisfies the $C$-monotonicity property
\begin{equation*}
\psi(\lambda s) \geq C^{-1}   \lambda^{1-\theta}  \psi(s), \quad
\forall s\in \R^+,\,\forall \lambda \geq 1.
\end{equation*}
}
\end{remark}

In order to state the next proposition we introduce the following
assumption
\begin{itemize}
\item[(b)] $\tilde b(t)\in C^1(\R_0^+)$, $\tilde b(t)>0$, $\tilde b'(t)\leq
0$ for $t\gg 1$, and $\tilde b^\lambda \not \in L^1(+\infty) $ for
some $\lambda >0$.
\end{itemize}

\begin{proposition}
\label{prop3.7}
Assume that conditions ($\Phi_0$), ($F_1$), ($L_1$), ($L_2$)
($\vp\ell$)$_1$, ($\theta$), (\ref{KO}) hold, and let $\tilde b$ a function
satisfying assumption (b),  $A>0,$  and $\beta \in [-2,+\infty)$.
If $\lambda$ and $\theta$ are the constants specified in
(b) and ($\theta$), assume also that
\begin{equation}
\label{eq3.8}
\begin{split}
&\lambda (2-\theta)\geq 1 \text{ and }\\
\text{ either } &\rmi t^{\beta/2} \tilde b(t)^{\lambda(1-\theta)-1}
\int_1^t\tilde b(s)^\lambda ds \leq C
\,\text{ for }\,  t\geq t_o\\
\text{ or } & \rmii t^{\beta/2} \tilde b(t)^{\lambda(1-\theta)-1} \leq C\
\text{ for } t\geq t_o \, \text{ and }\, \theta <1.\\
\end{split}
\end{equation}
Then there exists $T>0$ sufficiently large such that, for every $T\leq
t_0<t_1$ and $0<\epsilon<\eta$, there exist $\bar T>t_1$ and a $C^2$
function $\alpha : [t_0, \bar T)\to [\epsilon, +\infty)$ which
is a solution of the problem
\begin{equation}
\label{eq3.9}
\begin{cases}
\vp'(\alpha ') \alpha '' + A t^{\beta/2} \vp(\alpha ') \leq \tilde
b(t) f(\alpha) \ell (\alpha) \quad \text{on }\, [t_0, \bar T)&\\
\alpha '>0 \text{ on } [t_0, \bar T), \,\, \alpha(t_0) = \epsilon,
\, \, \alpha(t)\to +\infty \text{ as } t\to \bar T^-&
\end{cases}
\end{equation}
and satisfies
\begin{equation}
\label{eq3.10}
\epsilon \leq \alpha \leq \eta \quad \text{ on } \, [t_0, t_1].
\end{equation}
\end{proposition}

\begin{proof}
Note first of all, that the first condition in (\ref{eq3.8}) forces $\theta<2$, and
($\vp\ell$)$_2$ follows from ($\theta$)$_1$.

We choose $T>0$ large enough that, by (b), $\tilde b(t)>0$ and $\tilde b'(t)\leq
0$ on $[T,+\infty)$. Since (b) and (\ref{eq3.8}) are invariant under scaling
of $\tilde b$,  we may assume without loss of generality that $\tilde b\leq 1$ on $[T,\infty)$.

Let $t_0,$ $t_1$ $\epsilon$ $\eta$ be as in the
statement of the proposition,  and, for a given $\sigma\in (0, 1]$,
set
\begin{equation}
\label{eq3.11}
C_\sigma = \int_\epsilon^{+\infty} \frac{ds}{K^{-1}(\sigma F(s))},
\end{equation}
which is well defined in view of (\ref{KO}) and
Lemma~\ref{lemma_KOsigma}. Since $\tilde b(t)\not\in
L^{1}(+\infty)$, there exists $T_\sigma> t_o$ such that
\begin{equation*}
C_\sigma = \int_{t_0} ^{T_\sigma} \tilde b(s)^\lambda ds.
\end{equation*}
We note that, by monotone convergence, $C_\sigma \to +\infty$ as $\sigma \to
0+,$ and we may therefore choose $\sigma>0$ small enough that $T_\sigma
>t_1$. We let $\alpha:[t_0,T_\sigma) \to [\epsilon, +\infty)$  be
implicitly defined by the equation
\begin{equation}
\label{eq3.12}
\int_{t}^{T_\sigma}\tilde b(s)^\lambda ds =
\int_{\alpha(t)}^\infty \frac{ds}{K^{-1}(\sigma F(s))},
\end{equation}
so that, by definition,
\begin{equation*}
\alpha(t_0)= \epsilon,\quad \alpha (t)\to +\infty \text{ as } t\to
T_\sigma-.
\end{equation*}
Differentiating (\ref{eq3.12}) yields
\begin{equation}
\label{eq3.13}
\alpha '(t) = \tilde b(t) ^\lambda K^{-1}(\sigma F(\alpha(t))),
\end{equation}
so that $\alpha '>0$ on $[t_0, T_\sigma)$, and
\begin{equation*}
\sigma F(\alpha) = K(\alpha'/\tilde b^{\lambda}).
\end{equation*}
Differentiating once
more, using the definition of $K$ and (\ref{eq3.13}),  we obtain
\begin{equation}
\label{eq3.14}
\sigma f(\alpha)\alpha' = K'(\alpha'/\tilde b^\lambda) (\alpha
'/\tilde b^\lambda)'= \frac{\alpha'}{\tilde b^\lambda} \frac
{\vp'(\alpha'/\tilde b^\lambda)}{\ell(\alpha'/\tilde b^\lambda)}
\Bigl(\frac{\alpha'}{\tilde b^\lambda}\Bigr)'.
\end{equation}
Since $f(t)> 0$ on $(0,\infty)$, $\alpha'>0$ and $\tilde b'\leq 0,$
we have $(\alpha'/\tilde b^\lambda)'\geq 0$ and $\alpha'/\tilde
b^\lambda$ is non-decreasing.
Moreover,
\begin{equation*}
\Bigl(\frac{\alpha'}{\tilde b^\lambda}\Bigr)' =
(\alpha''/\tilde b^\lambda) - \lambda (\alpha'\tilde b'/\tilde
b^{\lambda +1})\geq (\alpha''/\tilde b^\lambda).
\end{equation*}
Inserting this into (\ref{eq3.14}), using the fact that $\tilde b^{-\lambda}\geq 1$
and ($\theta$)$_1$ (in the form of (\ref{prop3.2_eq1})), and rearranging we obtain
\begin{equation}
\label{eq3.15}
\vp'(\alpha ') \alpha'' \leq \Bigl\{C\sigma \tilde
b^{\lambda(2-\theta)}\Bigr\} \tilde b f(\alpha) \ell (\alpha'),\quad \text{on }\,
[t_0, T_\sigma).
\end{equation}
In order to estimate the term $At^{\beta/2} \vp(\alpha')$ we
rewrite (\ref{eq3.14}) in the form
\begin{equation*}
\vp(\alpha'/\tilde b^\lambda)\bigl({\alpha'}/{\tilde b^\lambda}\bigr)'
= \sigma \tilde b^{\lambda} f(\alpha) \ell(\alpha'/\tilde
b^\lambda),
\quad \text{on } [t_0, T_\sigma),
\end{equation*}
integrate between $t_0$ and $t\in (t_0,T_\sigma)$,
use the fact that $\alpha$, and $\alpha/\tilde b^\lambda$ are increasing, and
$f$ and $\ell$ are $C$-increasing
to deduce that
\begin{equation*}
\vp(\alpha'/\tilde b^\lambda)\leq  \vp(\alpha'/\tilde
b^\lambda)(t_0)+ C\sigma f(\alpha) \ell(\alpha'/\tilde
b^\lambda) \int_{t_0}^t \tilde b(s)^\lambda \ ds,
\end{equation*}
for some constant $C\geq 1$. On the other hand,
since $t^{\theta -1} \vp(t)/\ell(t)$ is $C$-increasing and $\tilde b\leq
1$,
we have
\begin{equation}
\label{eq3.19}
\begin{split}
\frac{\vp(\alpha')}{\ell(\alpha')} &\leq C\tilde b^{\lambda
(1-\theta)} \frac{\vp(\alpha'/\tilde b^\lambda)}{\ell(\alpha'/\tilde
b^\lambda)}
\\&
\leq C \tilde b^{\lambda (1-\theta)}
\Bigl[\frac{\vp(\alpha'/\tilde b^\lambda)(t_0)}{\ell(\alpha'/\tilde
b^\lambda)} + \sigma f(\alpha)\int_{t_0}^t\tilde b(s)^\lambda
\Bigr]\\
&\leq
C \tilde b^{\lambda (1-\theta)-1}
\Bigl[\frac{\vp(\alpha'/\tilde b^\lambda)(t_0)}{f(\epsilon) \ell(\alpha'/\tilde
b^\lambda)(t_0)} + \sigma\int_{t_0}^t\tilde b(s)^\lambda
\Bigr] \tilde b f(\alpha).
\end{split}
\end{equation}
where the second inequality follows from the fact that
$\alpha$ and $\alpha'/\tilde b^\lambda$ are increasing,
and $f$ and  $\ell$ are $C$-increasing.

Using  (\ref{eq3.15}) and (\ref{eq3.19}), and recalling that, by
(\ref{eq3.13}), $(\alpha'/\tilde b^\lambda)(t_0) =  K^{-1}(\sigma
F(\epsilon))$, we obtain
\begin{equation}
\label{eq3.20}
\vp'(\alpha') \alpha'' + A t^{\beta/2} \vp(\alpha')
\leq N_\sigma (t) \tilde b f(\alpha )\ell(\alpha'),
\end{equation}
where
\begin{multline}
\label{eq3.20bis}
N_\sigma (t)=
C \sigma \tilde b^{\lambda (2-\theta) - 1}
+  AC t^{\beta/2} \tilde b^{\lambda (1-\theta) -1}
\frac{\vp(K^{-1}(\sigma F(\epsilon)))}{\ell(K^{-1}(\sigma F(\epsilon)))f(\epsilon)}
\\
+ AC \sigma t^{\beta/2} \tilde b^{\lambda (1-\theta) -1}\int_{t_0}^t\tilde b(s)^\lambda
= (I)(t)+(II)(t)+(III)(t).
\end{multline}
Since $\tilde b\leq 1$, and $\lambda(2-\theta)-1\geq 0$ by
(\ref{eq3.8}), we see that
\begin{equation*}
(I)(t)\to 0\, \text{ uniformly on }\, [t_0,+\infty) \, \text{ as }\,
\sigma\to 0.
\end{equation*}
As for $(II)$, according to (\ref{eq3.8})
\begin{equation*}
t^{\beta/2} \tilde b^{\lambda (1-\theta) -1}\leq C
\, \text { on }\, [t_o, +\infty),
\end{equation*}
so that, using ($\phi\ell$)$_1$, we deduce that
\begin{equation*}
\liminf_{\sigma \to 0+}
\frac{\vp(\widehat K^{-1}(\sigma F(\epsilon)))}{f(\epsilon) \ell(\widehat K^{-1}(\sigma
F(\epsilon)))} = 0.
\end{equation*}
Thus
\begin{equation*}
 (II)(t)\to 0 \, \text{ uniformly on }\, [t_0, +\infty)\,\text{ along a sequence }\,
 \sigma_k \to
 0.
\end{equation*}
It remains to analyze $(III)$. Clearly, if (\ref{eq3.8}) \rmi
holds, then $(III)(t)\to 0$ uniformly on $[t_0,+\infty)$ as $\sigma\to
0.$ Assume therefore that (\ref{eq3.8}) \rmii holds, so that
\begin{equation}
(III)(t) \leq AC \sigma \int_{t_0}^t \tilde b(s)^\lambda ds.
\end{equation}
By the definition of $\alpha (t)$, Proposition~\ref{prop 3.2bis},
and $(KO)$
\begin{equation*}
\begin{split}
\int_{t_0}^t \tilde b(s)^\lambda ds
&= \int_\epsilon ^{\alpha(t)}\frac {ds} {K^{-1}(\sigma F(s))}\\
&\leq B \sigma^{-1/(2-\theta)}
\int_{\epsilon}^{+\infty}\frac {ds} {K^{-1}( F(s))}
\leq C \sigma^{-1/(2-\theta)},
\end{split}
\end{equation*}
in $[t_0,T_\sigma)$. Since $\theta <1$ we conclude that
\begin{equation*}
(III)(t) \leq C \sigma^{1-1/(2-\theta)} \to 0 \, \text{ uniformly in
}\, [t_0, T_\sigma) \, \text{ as }\, \sigma \to 0.
\end{equation*}
Putting  together the above estimates, we conclude that we can choose
$\sigma$ small enough that $N_\sigma (t)\leq 1$, showing that
$\alpha(t)$ satisfies the differential inequality in (\ref{eq3.9}).

In order to complete the proof we only need to prove
that $\epsilon\leq \alpha(t)\leq \eta$
for $t_0\leq t\leq t_1$. Again from the definition of $\alpha$ we have
\begin{equation*}
\int_{t_0}^{t_1} \tilde b(s)^\lambda ds =
\int_\epsilon ^{\alpha(t_1)} \frac{ds}{K^{-1}(\sigma F(s))},
\end{equation*}
so if we choose $\sigma \in (0,1]$ small enough to have
\begin{equation*}
\int_{t_0}^{t_1} \tilde b(s)^\lambda ds \leq
\int_\epsilon ^{\eta} \frac{ds}{K^{-1}(\sigma F(s))},
\end{equation*}
then clearly $\alpha(t_1)\leq \eta$, and, since $\alpha$ is
increasing, this finishes the proof.
\end{proof}

We are now ready to prove

\begin{theorem}
\label{thm3.21}
Let $(M,\langle  \,,\rangle)$ be a complete Riemannian manifold
satisfying
\begin{equation}
\tag{\ref{genRicci_lower_bound}}
\mathrm{Ricc}_{n,m} (L_D) \geq H^2(1+r^2)^{\beta/2},
\end{equation}
for some $n>m$, $H>0$ and $\beta\geq -2$ and assume  that ($\Phi_0$), ($F_1$),
($L_1$), ($L_2$), ($\vp\ell$)$_1$, and ($\theta$) hold. Let
$b(x)\in C^0(M)$, $b(x)\geq 0$ on $M$ and suppose that
\begin{equation}
\label{eq3.22}
b(x)\geq  \tilde  b(r(x)) \quad \text{for } \, \, r(x)\gg 1,
\end{equation}
where $\tilde b$ satisfies assumption ($b$) and (\ref{eq3.8}).
If the Keller--Osserman condition
\begin{equation}
\tag{KO} \frac 1 {K^{-1} (F(t))} \in L^1(+\infty)
\end{equation}
holds then any entire classical weak solution $u$ of the
differential inequality
\begin{equation}
\tag{\ref{main_ineq}}
L_{D,\vp} u \geq b(x) f(u) \ell(\modnabla u)
\end{equation}
is either non-positive or constant. Furthermore, if $u\geq 0,$ and $\ell(0)
>0$, then $u$ vanishes identically.
\end{theorem}

\begin{proof}
If $u\leq 0$ then there is nothing to prove.
We argue by contradiction and assume that $u$ is non-constant and positive somewhere.
We choose $T>0$ sufficiently large that (\ref{eq3.22}) holds in $M\setminus B_T$
and for every $r_o\geq T$ we have
\begin{equation*}
0<u^*_o = \sup_{B_{r_o}} u \leq u^* = \sup_M u.
\end{equation*}
We consider first the case where $u^*<+\infty$. We claim that $u^*_{o}<
u^*$. Otherwise there would exists $x_o\in \overline{B}_{r_o}$ such
that $u(x_o)= u^*,$ and by (\ref{main_ineq}) and assumptions ($F_1$)
and ($\ell_1$),
\begin{equation*}
L_{D,\vp} u\geq 0
\end{equation*}
in the connected component $\Omega_{o}$ of $\{u\geq 0\}$ containing $x_o$. By the
strong maximum principle \cite{PucciSerrin-Book}, $u$ would then be  constant and positive
on $\Omega_o$. Since $u=0$ on $\bdr \Omega_o$ this would imply that
$\Omega_o=M$ and $u$ is a positive constant on $M$, contradicting
our assumption.

Next, we choose $\eta>0$ small enough that $u^*_o +2\eta <u^*$ and
$\tilde x \not\in \overline {B}_{r_o}$ satisfying $u(\tilde x) > u^* -\eta.$
We let $t_o=r_o$ and $t_1 = r(\tilde x)$. Because of
(\ref{genRicci_lower_bound}), Proposition~\ref{weighted_laplacian_comparison} and
\cite{PigolaRigoliSetti-TheBook}, Proposition 2.11, there exists $A=A(T)>0$ such that
\begin{equation*}
L_D r \leq A r^{\beta/2}
\quad \text{on }\, M\setminus B_T.
\end{equation*}
According to Proposition~\ref{prop3.7} there exist
$\bar T>t_1$ and a $C^2$ function $\alpha : [t_0, \bar T)\to [\epsilon, +\infty)$ which
satisfies
\begin{equation*}
\begin{cases}
\vp'(\alpha ') \alpha '' + A t^{\beta/2} \vp(\alpha ') \leq (2C)^{-1} \tilde
b(t) f(\alpha) \ell (\alpha) \quad \text{on }\, [t_0, \bar T)
&\\
\alpha '>0 \text{ on } [t_0, \bar T), \,\, \alpha(t_0) = \epsilon,
\, \, \alpha(t)\to +\infty \text{ as } t\to \bar T^-
&\\
\end{cases}
\end{equation*}
and
\begin{equation*}
\epsilon \leq \alpha \leq \eta \quad \text{ on } \, [t_0, t_1],
\end{equation*}
where $C$ is the constant in the definition of $C$-monotonicity of
$f.$

It follows  that  the radial function defined on $B_{\bar R}\setminus B_{r_o}$
by $v(x) = \alpha (r(x))$ satisfies the differential inequality
\begin{equation}
\label{eq3.2bis}
L_{D,\vp} v\leq  (2C)^{-1} b(x) [f(\alpha)\ell(\alpha')](r(x)).
\end{equation}
pointwise in $(B_{\bar R}\setminus \overline{B}_{r_o})\setminus
\text{cut}(o)$ and weakly in $B_{\bar R}\setminus
\overline{B}_{r_o}$.
Furthermore $v$ satisfies (\ref{eq3.1}), and
\begin{equation*}
u(\tilde x) - v(\tilde x) > u^* -2\eta.
\end{equation*}
Since
\begin{equation*}
u(x) - v(x) \leq u^*_o - \epsilon < u^* - 2\eta -\epsilon \quad
\text{ on } \, \bdr B_{r_o}
\end{equation*}
and
\begin{equation*}
 u(x) - v(x) \to -\infty \quad \text{ as } \,  x\to \bdr B_{\bar R},
\end{equation*}
we deduce that the function $u-v$ attains a positive maximum $\mu$
in $B_{\bar R} \setminus \overline{B}_{r_o}$. We denote be
$\Gamma_\mu$ a connected component of the set
\begin{equation*}
\{ x\in B_{\bar R} \setminus \overline{B}_{r_o} \, :\, u(x)-v(x) = \mu\}
\end{equation*}
and note that $\Gamma_\mu$ is compact.

We claim that for every $y \in\ \Gamma_\mu$
we have
\begin{equation}
\label{eq3.2ter}
u(y) > v(y), \quad |\nabla u(y)|=|\alpha'(r(y))|.
\end{equation}
Indeed, this is obvious if $y$ is not in the cut locus $\text{cut}(o)$ of
$o$, for then $\nabla u (y)= \nabla v (y)= \alpha'(r(y))\nabla r
(y)$. On the other hand, if $y\in \text{cut}(o)$, let $\gamma$ be
a unit speed minimizing geodesic joining $o$ to $y$, let
$o_\epsilon = \gamma (\epsilon)$ and let $r_\epsilon (x)= d(x, o_\epsilon).$
By the triangle inequality,
\begin{equation*}
r(x)\leq r_\epsilon(x) + \epsilon \,\quad, \forall x\in M,
\end{equation*}
with equality if and only if $x$ lies on the portion of the geodesic $\gamma$
between $o_\epsilon$ and $y$ (recall that $\gamma$ ceases to be minimizing
past $y$). Define $v_\epsilon (x) = \alpha (\epsilon + r_\epsilon (x))$,  then, since
$\alpha$ is strictly increasing,
\begin{equation*}
v_\epsilon(x)\geq v(x)
\end{equation*}
with equality if and only if $x$ lies on the portion of
$\gamma$ between $o_\epsilon$ and $y$.
We conclude that $\forall x \in B_{R}\setminus \overline{B}_{r_o}$,
\begin{equation*}
(u-v_\epsilon)(y) = (u-v)(\xi) \geq (u-v)(x) \geq
(u-v_\epsilon)(x),
\end{equation*}
and $u-v_\epsilon$ attains a maximum at $y.$ Since $y$ is not on
the cut locus of $o_\epsilon$, $v_\epsilon$ is smooth there, and
\begin{equation*}
|\nabla u(y)| = |\nabla v_\epsilon(y)| = \alpha'(\epsilon +
r_\epsilon (y) ) |\nabla r_\epsilon(y)| = \alpha'(r(y)),
\end{equation*}
as claimed.

Since $f$ is $C$-increasing,
\begin{equation*}
b(y) f(u(\xi)) \ell(\modnabla u (y)) \geq  \frac 1C b(y) f(v(y))
\ell(\alpha'(r(y)))
\end{equation*}
and by continuity the inequality
\begin{equation*}
b(x) f(u) \ell(\modnabla u) \geq \frac 1{2C} b(x) f(v(x))
\ell(\alpha'(r(x)))
\end{equation*}
holds in a neighborhood of $y$.
It follows from this and the differential inequalities satisfied by
$u$ and $v$ that
\begin{equation}
\label{eq3.23'}
L_{D,\vp} u\geq L_{D,\vp} v
\end{equation}
weakly in a sufficiently small neighborhood $\mathcal{U}$ of
$\Gamma_\mu$.
Now fix $y\in \Gamma_\mu$ and for $\zeta \in (0,\mu)$ let
$\Omega_{y,\zeta}$ the connected component containing $y$ of the set
\begin{equation*}
\{x\in B_{\bar R} \setminus \overline{B}_{r_o}\,:\, u(x) >
v(x) +\zeta\}.
\end{equation*}
By choosing $\zeta$ sufficiently close to $\mu$ we may arrange that
$\overline \Omega_{y,\zeta} \subset \mathcal{U},$ and, since
$u=v+\zeta$ on $\bdr \Omega_{y,\zeta}$, (\ref{eq3.23'}) and the weak
comparison principle (see, e.g., \cite{PigolaRigoliSetti-Memoirs}, Proposition 6.1)
implies that $u\leq v +\zeta$ on $\Omega_{y,\zeta}$, contradicting the fact that $y\in
\Omega_{y,\zeta}$.

The case where $u^*  = +\infty$ is easier, and left to the reader.
\end{proof}

\begin{remark}
\label{proof ThmA}
{\rm
Theorem~A is a special case of Theorem~\ref{thm3.21}
with the choice $\tilde b(r) = C/r^\mu$ for $r\gg 1$. Assume first that $\mu>0$.
Choosing $\lambda
=1/\mu$, it follows that
\begin{equation*}
t^{\beta/2} \tilde b(t)^{\lambda(1-\theta)-1}
= O(t^{\theta -1 + \beta/2 +\mu})\, \text{ and }\,
 \int_1^t \tilde b(s)^{\lambda} ds = O(\log t).
\end{equation*}
Then (\ref{thetabetamu}) (and $\beta\geq -2$) implies  first that $\lambda(2-\theta)-1
\geq \mu^{-1}(1+\beta/2)\geq 0$, and then that either \rmi or \rmii in (\ref{eq3.8}) holds.
Thus  Theorem~\ref{thm3.21} applies. On the other hand, if $\mu=0$  and $\theta< 1-\beta/2$, then
$\theta < 1 - \beta/2-\mu_o$ for sufficiently small $\mu_o>0$, and
the conclusion follows from the previous case.
}
\end{remark}

The next example shows that the validity of the generalized
Keller--Osserman condition (KO) is indeed necessary for
Theorem~\ref{thm3.21} to hold. Since (KO) in independent of geometry,
we consider the most convenient setting where $(M, \langle
\,,\rangle)$ is $\R^m$ with its canonical flat metric. We further
simplify our analysis by considering the differential inequality
\begin{equation}
\label{eq3.23} \Delta_p u\geq f(u)\ell(\modnabla u),
\end{equation}
for
the $p$-Laplacian $\Delta_p$, where $f$ is increasing and satisfies
$f(0)=0$ $f(t)>0$ for $t>0$, $\ell$ is non-decreasing and satisfies ($L_1$),
and ($\vp\ell$) and $\theta$ hold. We let $K:\R^+_0\to \R^+_0$ be
defined as in (\ref{K_def}), and assume that
\begin{equation}
\label{notKO}
\tag{$\neg$KO}\frac 1{K^{-1}(F(t))}\not\in L^1(+\infty).
\end{equation}
Define implicitly the function $w$ on $\R^+_0$
by setting
\begin{equation}
\label{eq3.24}
t=\int_1^{w(t)}
\frac{ds}{K^{-1}(F(s))}.
\end{equation}
Note that $w$ is well defined, $w(0)=1$, and
($\neg$KO) and 
imply that $w(t)\to
+\infty$ as $t\to \infty.$
Differentiating (\ref{eq3.24}) yields
\begin{equation}
\label{eq3.25}
 w'= K^{-1}( F(w(t))) >0,
\end{equation}
and a further differentiation gives
\begin{equation}
\label{eq3.26}
(w')^{p-2} w'' = \frac 1{p-1}
f(w)\ell(\modnabla{w}).
\end{equation}
We fix $\bar t>0$ to be specified later, and let $u_1(x)$ be the
radial function defined on $\R^m\setminus B_{\bar t}$ by the
formula
\begin{equation*}
u_1(x) = w(|x|).
\end{equation*}
Using (\ref{eq3.25}) and (\ref{eq3.26}) we conclude that $u_1$
satisfies
\begin{equation}
\label{eq3.27}
\Delta_p u_1 = (p-1) (w')^{p-2} w''+ \frac{m-1}{|x|}
(w')^{p-1}\geq  f(u_1)\ell(\modnabla{u_1})
\end{equation}
on $\R^m\setminus \overline{B}_{\bar t}$.

Next we fix constants $\beta_o$, $\Lambda>0$, and, denoting with $p'$
the conjugate exponent of $p$, we let
\begin{equation*}
\beta(t) = \frac \Lambda {p'} t^{p'} + \beta_o.
\end{equation*}
Noting that $\beta'(0)=0,$ we deduce that the function
\begin{equation*}
u_2(x) = \beta(|x|)
\end{equation*}
is $C^1$ on $\R^m$, and an easy calculation shows that
\begin{equation}
\label{eq3.28} \Delta_p u_2 = \Lambda^{p-1}\div{|x| x} = m\Lambda^{p-1}.
\end{equation}
Since $ \beta'\geq 0,$ and $f$ and $\ell$ are monotonic, it
follows that, if
\begin{equation}
\label{eq3.29}
m\Lambda^{p-1}\geq f(\beta(\bar t) \ell(\beta'(\bar t)),
\end{equation}
then
\begin{equation}
\label{eq3.30}
 \Delta_p u_2 \geq f(u_2) \ell(\modnabla{u_2}) \quad\text{ on } \,
 B_{\bar t}.
\end{equation}
The point now is to join $u_1$ and $u_2$ in such a way that the
resulting function u is a classical $C^1$ weak subsolution of
\begin{equation*}
\Delta_p u = f(u) \ell(\modnabla u).
\end{equation*}
This is achieved provided  we may choose the parameters $\bar t$,
$\Lambda$, $\beta_o$, in such a way that (\ref{eq3.29}) and
\begin{equation}
\label{eq3.31}
\begin{cases}
\beta(\bar t) = w(\bar t)&\\
\beta'(\bar t) = w'(\bar t)&
\end{cases}
\end{equation}
are satisfied.
Towards this end, we define
\begin{equation}
\label{eq3.32} \bar t = \int_1^\lambda \frac{ds}{K^{-1}(F(s))} > 0,
\end{equation}
where  $1< \lambda \leq 2$. Note that, by definition,  $w(\bar
t)=\lambda$, and, by the monotonicity of $K^{-1}$ and $F$
\begin{equation}
\label{eq3.32bis}
\frac{\lambda -1 }{ K^{-1}(F(2)) } \leq \bar t \leq \frac{\lambda -1 }{ K^{-1}(F(1))
},
\end{equation}
so that, in particular, $\bar t \to 0$ as $\lambda \to 1^+$.
Putting together (\ref{eq3.29}) and (\ref{eq3.31}) and recalling the relevant
definitions we need to show that the  following system of
inequalities
\begin{equation}
\label{eq3.34}
\begin{cases}
\rmi\,\,\,\,   K^{-1}(F(\lambda))\bar t/p' + \beta_o = \lambda&\\
\rmii
 \,\, \Lambda \bar t^{p'-1} = K^{-1} (F(\lambda)) &\\
\rmiii \,m\Lambda ^{p-1} \geq f(\lambda) \ell (K^{-1}(F(\lambda))).
&
\end{cases}
\end{equation}
Since, by (\ref{eq3.32bis}),
\begin{equation*}
K^{-1}(F(\lambda))\frac{\bar t}{p'} \leq   \frac 1{p'} \frac
{K^{-1}(F(2))}{K^{-1}(F(1))}(\lambda -1)
\end{equation*}
for $\lambda$ sufficiently close to $1$ the first summand on
the left hand side of {\rmi} is strictly less that $1$, and therefore we may choose
$\beta_o >0$ in such a way that {\rmi} holds.
Next we let $\Lambda$ be defined by {\rmii}, and  note that,
\begin{equation*}
\Lambda = K^{-1}(F(\lambda)) \bar t^{1-p'} \geq K^{-1}(F(1))\to
+\infty \quad \text{ as }\, \lambda \to1^+.
\end{equation*}
Therefore, since
\begin{equation*}
f(\lambda) \ell(K^{-1} (F(\Lambda))) \leq f(2)\ell(K^{-1}(F(2))),
\end{equation*}
if $\lambda$  is close enough to $1$ then {\rmiii} is also
satisfied.

Summing up, if $\lambda $ is sufficiently close to $1$, the function
\begin{equation}
u(x) =
\begin{cases}
u_1(x) &\text{on } \, \R^m\setminus B_{\bar t}\\
u_2(x) &\text{on } \, B_{\bar t}
\end{cases}
\end{equation}
is a classical weak solution of (\ref{eq3.23}).

We remark that we may easily arrange that assumptions ($\vp\ell$)
and ($\theta$) are also satisfied. Indeed, if we choose, for
instance, $\ell(t) = t^q$ with $q\geq 0,$ then, as already noted in the
Introduction, ($\vp$) holds for every $p>1+q$ and ($\theta$) is
verified for every $\theta\in \R$ such that $p\geq 2+q-\theta.$

We also stress that the solution $u$ of (\ref{eq3.23}) just constructed
is positive and diverges at infinity. Indeed the method used in the
proof of Theorem~\ref{thm3.21} may be adapted to yield non-existence of
non-constant, non-negative bounded solutions even when ($\neg$KO)
holds. This is the content of the next

\begin{theorem}
\label{thm3.35}
Maintain notation and assumptions of Theorem~\ref{thm3.21}, except for (KO)
which is replaced by ($\neg$KO). Then any non-negative, bounded,
entire classical weak solution $u$ of the differential inequality
(\ref{main_ineq}) is constant. Furthermore, if $\ell(0)>0$, then $u$
is identically zero.
\end{theorem}

The proof of the theorem follows the lines of that of
Theorem~\ref{thm3.21} once we prove the following

\begin{proposition}
\label{prop3.36}
In the assumptions of Proposition~\ref{prop3.7}, with (KO) replaced
by ($\neg$KO), there exists $T>0$ large enough that for every $T\leq
t_0<t_1,$ and $0<\epsilon<\eta$, there exists a $C^2$ function
$\alpha:[t_0,+\infty)  \to [\epsilon, +\infty)$ which solves the
problem
\begin{equation}
\label{eq3.37}
\begin{cases}
\vp'(\alpha ') \alpha '' + A t^{\beta/2} \vp(\alpha ') \leq \tilde
b(t) f(\alpha) \ell (\alpha) \quad \text{on }\, [t_0, \bar T)&\\
\alpha '>0 \text{ on } [t_0, \bar T), \,\, \alpha(t_0) = \epsilon,
\, \, \alpha(t)\to +\infty \text{ as } t\to +\infty&
\end{cases}
\end{equation}
and satisfies
\begin{equation}
\label{eq3.38}
\epsilon \leq \alpha \leq \eta \quad \text{ on } \, [t_0, t_1].
\end{equation}
\end{proposition}
\begin{proof}
The argument is similar to that of Proposition~\ref{prop3.7}. The
main difference is in the definition of $\alpha$ which now proceeds
as follows. We fix $T>0$ large enough that (b) holds on $[\bar T,
+\infty).$ For $t_0,t_1, \epsilon, \eta$ as in the statement, and $\sigma\in
(0,1]$ we implicitly define $\alpha: [t_0, +\infty) \to [\epsilon, +\infty)$
by setting
\begin{equation*}
\int_{t_0}^t \tilde b(s)^\lambda ds =
\int_\epsilon^{\alpha(t)} \frac{ds}{K^{-1}(\sigma F(s))},
\end{equation*}
so that $\alpha(t_0) = \epsilon,$ and, by (b) and ($\neg$KO), $\alpha(t)\to
+\infty$ as $t\to +\infty.$ The rest of the proof proceeds as in
Proposition~\ref{prop3.7}.
\end{proof}
Summarizing, the differential inequality (\ref{main_ineq}) may admit
non-constant, non-negative entire classical weak solutions only if
($\neg$KO) holds, and possible solutions are necessarily unbounded.
We shall address this case in Section~\ref{section weak_maximum}

\section{A further version of Theorem A}
\label{section_further_version}
As mentioned in the Introduction, condition ($\vp\ell$) fails, for instance, when
$\vp$ is of the form
\begin{equation*}
\vp(t) =\frac t{\sqrt{1+t^2}}
\end{equation*}
which, when $D(x)\equiv 1$, corresponds to the mean curvature
operator. Because of the importance of this operator, in Geometry
as well as in Analysis, it is desirable to have a version of
Theorem~A valid when ($\vp\ell$)$_2$ fails. To deal with this
situation we consider an alternative form of the Keller--Osserman
condition, and correspondingly, modify our set of assumptions. We
therefore replace assumption ($\vp\ell$)$_2$ with
\begin{itemize}
\item [($\Phi_2$)] There exists $C>0$ such that $\vp(t)\geq Ct\vp'(t)$
on $\R^+.$
\item [($\vp\ell$)$_3$] $\frac{\vp(t)}{\ell(t)} \in L^1(0^+)\setminus
L^1(+\infty)$.
\end{itemize}
As noted in Remark~\ref{rmk_theta_vpell}, ($\vp\ell$)$_3$ is implied by ($\theta$)$_2$ with $\theta<2$

It is easy to verify that in the case of the mean curvature
operator,
\begin{equation*}
t\vp'(t)= \frac t {(1+t^2)^{3/2} }\leq \vp (t) \text{ and  }
\vp(t)\sim
\begin{cases}
t&\text{ as } \to 0^+\\
1&\text{ as } t\to +\infty,
\end{cases}
\end{equation*}
so that ($\Phi_2$) holds, and $(\vp\ell$)$_3$ is satisfied provided
$t\ell^{-1}\in L^1(0^+)$ and $\ell^{-1} \not\in L^1 (+\infty)$. By contrast, the choice
\begin{equation*}
\vp(t) = te^{t^2},
\end{equation*}
corresponding to the operator of exponentially harmonic functions,
does not satisfy ($\Phi_2$).

According to ($\vp\ell$)$_3$, we may define a function $\widehat K$
by
\begin{equation}
\label{eq4.0}
\widehat K(t) =\int _0^t \frac{\vp(s)}{\ell(s)} ds
\end{equation}
which is well defined on $\R^+_0$, tends to $+\infty$ as $t\to +\infty$
and therefore  gives rise to a $C^1$ diffeomorphism of $\R^+_0$
on to itself.

The variant of the generalized Keller--Osserman condition mentioned
above is then
\begin{equation}
\tag{\^{K}O}
\frac 1 {\widehat K^{-1}(F(t))} \in L^1(+\infty).
\end{equation}

Analogues of Lemma~\ref{lemma_KOsigma},  Proposition~\ref{prop 3.2bis} and Proposition~\ref{prop3.7}
are also valid in this setting.
\begin{lemma}
\label{lemma_KOsigma_bis}
Assume that $f$, $\ell$ and $\vp$ satisfy the assumptions ($F_1$),
($L_1$) and ($\vp\ell$)$_3$, and let $\sigma>0$. Then ({\^{K}O}) holds
if and only if
\begin{equation}
\tag{\^KO$\sigma$}
\frac 1{\widehat K^{-1}(\sigma F(s))}\in L^1(+\infty).
\end{equation}
\end{lemma}

Indeed, the proof of Lemma~\ref{lemma_KOsigma} depends only on the monotonicity of $K$ and the
$C$-monotonicity of $f$, and can be repeated without change replacing $K$ with $\widehat K$.

Similarly, using Remark~\ref{rmk prop 3.2bis}, one establishes the following

\begin{proposition}
\label{prop 3.2ter}
Assume that conditions ($\Phi_0$) and ($L_1$) hold, and let $F$ be a positive function
defined on $\R_0^+$. If ($\theta$)$_2$ holds with $\theta<2$, then there
exists a constant  $B>1$ such that, for every $\sigma \leq 1$ we have
\begin{equation}
\label{eq_prop3.2ter_1}
\frac{\sigma^{1/(2-\theta)}}{\widehat K^{-1}(\sigma F(t))}\leq
\frac B{\widehat K^{-1}(F(t))}
\quad \text{ on } \R^+.
\end{equation}
\end{proposition}

Finally, we have
\begin{proposition}
\label{prop4.2}
Assume that ($\Phi_0$), ($\Phi_2$),
($F_1$), ($L_1$), ($L_2$), ($\vp\ell$)$_1$,
($\theta$)$_2$ and (\^{K}O) hold,
let $\tilde b$ a function satisfying assumption (b), and let  $A>0,$  and $\beta \in [-2,+\infty)$.
If $\lambda$ and $\theta$ are the constants specified in
(b) and ($\theta$), assume also that
\begin{equation}
\tag{\ref{eq3.8}}
\begin{split}
&\lambda (2-\theta)\geq 1 \text{ and }\\
\text{ either } &\rmi t^{\beta/2} \tilde b(t)^{\lambda(1-\theta)-1}
\int_1^t\tilde b(s)^\lambda ds \leq C
\,\text{ for }\,  t\geq t_o\\
\text{ or } & \rmii t^{\beta/2} \tilde b(t)^{\lambda(1-\theta)-1} \leq C\
\text{ for } t\geq t_o \, \text{ and }\, \theta <1.\\
\end{split}
\end{equation}
Then there exists $T>0$ sufficiently large such that, for every $T\leq
t_0<t_1$ and $0<\epsilon<\eta$, there exist $\bar T>t_1$ and a $C^2$
function $\alpha : [t_0, \bar T)\to [\epsilon, +\infty)$ which
is a solution of the problem
\begin{equation}
\tag{\ref{eq3.9}}
\begin{cases}
\vp'(\alpha ') \alpha '' + A t^{\beta/2} \vp(\alpha ') \leq \tilde
b(t) f(\alpha) \ell (\alpha) \quad \text{on }\, [t_0, \bar T)&\\
\alpha '>0 \text{ on } [t_0, \bar T), \,\, \alpha(t_0) = \epsilon,
\, \, \alpha(t)\to +\infty \text{ as } t\to \bar T^-&
\end{cases}
\end{equation}
and satisfies
\begin{equation}
\tag{\ref{eq3.10}}
\epsilon \leq \alpha \leq \eta \quad \text{ on } \, [t_0, t_1].
\end{equation}
\end{proposition}

\begin{proof}
The proof is a small variation of that of Proposition~\ref{prop3.7}, using $\widehat K$
instead of $K$ in the definition of $\alpha$.

Note first of all that (\ref{eq3.8}) forces $\theta<2$, so that
($\vp\ell$)$_3$ is automatically satisfied.

Arguing as in Proposition~\ref{prop3.7}, one deduces that $\alpha'>0$
and $\alpha$ satisfies
\begin{equation}
\label{eq4.11} \sigma f(\alpha) \alpha' =
\frac{\vp(\alpha'/\tilde b^\lambda)}{\ell(\alpha'/\tilde b^\lambda)}
\bigl(\alpha'/\tilde b^\lambda\bigr)',
\end{equation}
so, again, $\alpha'/\tilde b^\lambda$ is
 increasing on $[t_0, T_\sigma)$. From this, using the fact that
$t^{\theta-1}\vp(t)/\ell(t)$ is $C$-increasing (assumption ($\theta$)$_2$),
$\vp(t)\geq C t \vp'(t)$ (assumption ($\Phi_2$)), and $\tilde b(t)^{-\lambda}>1$, we obtain
\begin{equation}
\label{eq4.13}
\vp'(\alpha')\alpha ''\leq
\Bigl(C\sigma\tilde b^{\lambda (2-\theta)-1}\Bigr) bf(\alpha)
\ell(\alpha')
\end{equation}
on $[t_0,T_\sigma)$, for some  constant $C>0$. On the other
hand, applying ($\Phi_2$) to (\ref{eq4.11}), rearranging,
integrating  over $[t_0,t]$, and using ($F_1$), ($L_2$) and
the fact that $\alpha$ and $\alpha'/\tilde b^\lambda $ are
increasing, we deduce that
\begin{equation*}
\vp(\alpha'/\tilde b^\lambda)  \leq \vp(\alpha'/\tilde b^\lambda)(t_0) +
 C\sigma f(\alpha)\ell(\alpha'/\tilde b^\lambda)
\int_{t_0}^t\tilde b(s)^\lambda ds.
\end{equation*}
Finally, using ($F_1$), ($L_2$), the fact that $\alpha$ and $\alpha'/\tilde
b^\lambda$ are non-decreasing, $\alpha(t_0)=\epsilon$ and
($\theta$)$_2$ we obtain
\begin{equation}
\label{eq4.14}
\frac{\vp(\alpha')}{\ell(\alpha')} \leq
C \tilde b^{\lambda (1-\theta)-1}
\Bigl[\frac{\vp(\alpha'/\tilde b^\lambda)(t_0)}{f(\epsilon) \ell(\alpha'/\tilde
b^\lambda)(t_0)} + \sigma\int_{t_0}^t\tilde b(s)^\lambda
\Bigr] \tilde b f(\alpha).
\end{equation}
Combining (\ref{eq4.13}) and (\ref{eq4.14}) we conclude that
\begin{equation}
\label{eq4.15}
\vp'(\alpha')'\alpha'' +A t^{\beta/2} \vp(\alpha') \leq
N_\sigma \tilde b f(\alpha)\ell(\alpha')
\end{equation}
with $N_\sigma(t)$ defined as in (\ref{eq3.20bis}).

The proof now proceeds exactly as in the case of Proposition~\ref{prop3.7}
\end{proof}

We then have the following version of Theorem~\ref{thm3.21}:

\begin{theorem}
\label{thm4.1}
Let $(M,\langle  \,,\rangle)$ be a complete Riemannian manifold
satisfying
\begin{equation}
\tag{\ref{genRicci_lower_bound}}
\mathrm{Ricc}_{n,m} (L_D) \geq H^2(1+r^2)^{\beta/2},
\end{equation}
for some $n>m$, $H>0$ and $\beta\geq -2$ and assume  that ($\Phi_0$), ($\Phi_2$), ($F_1$),
($L_1$), ($L_2$), ($\vp\ell$)$_1$, ($\vp\ell$)$_2$ and ($\theta$)$_2$ hold. Let
$b(x)\in C^0(M)$, $b(x)\geq 0$ on $M$ and suppose that
\begin{equation}
\tag{\ref{eq3.22}}
b(x)\geq  \tilde  b(r(x)) \quad \text{for } \, \, r(x)\gg 1,
\end{equation}
where $\tilde b$ satisfies assumption ($b$) and (\ref{eq3.8}).
If the modified Keller--Osserman condition
\begin{equation}
\tag{\^{K}O} \frac 1 {\widehat{K}^{-1} (F(t))} \in L^1(+\infty)
\end{equation}
holds then any entire classical weak solution $u$ of the
differential inequality
\begin{equation}
\tag{\ref{main_ineq}}
L_{D,\vp} u \geq b(x) f(u) \ell(\modnabla u)
\end{equation}
is either non-positive or constant. Furthermore, if $u\geq 0,$ and $\ell(0)
>0$, then $u$ vanishes identically.
\end{theorem}

According to Remark~\ref{proof ThmA},
Theorem~\ref{thm4.1} holds if we assume that $\tilde b(t) = C/ t^\mu
$ for $t\gg 1$ where $\mu\geq 0$ and 
\begin{equation}
\tag{$\theta\beta\mu$}
\begin{cases}
\theta < 1-\beta/2 - \mu \,\, \text{ or } \,\,\theta = 1-\beta/2 -
\mu<1 &\text{if }\,  \mu>0\\
\theta< 1-\beta/2 &\text{ if } \, \mu=0.
\end{cases}
\end{equation}

We note that in the model case of the mean curvature operator
with
\begin{equation*}
\ell(t)= t^q, \, q\geq 0,
\end{equation*}
then assumptions ($\Phi_0$), ($\Phi_2$), ($\vp\ell$)$_1$ and
($\theta$)$_2$ hold provided
\begin{equation*}
(0\leq)q <1, \quad \theta \geq 1+q
\end{equation*}
and the above restrictions are compatible with ($\theta\beta\mu$).

\section{Weak Maximum Principle and Non-Existence of Bounded Solutions}
\label{section weak_maximum}
As shown in Section~\ref{section_proof_ThmA} above, the failure of
the Keller-Osserman condition,
allows to deduce existence of
solutions of the differential inequality (\ref{main_ineq}). The
solutions thus constructed diverge at infinity. This is no accident.
Indeed, Theorem~B shows that under rather mild conditions on
the coefficients and on the geometry of the manifold, if solutions
exist, they must be unbounded, and in fact, must go to infinity
sufficiently fast.

The proof of the Theorem~B depends on the following weak maximum principle
for the diffusion operator $L_{D,\vp}$ which improves on the weak maximum
principle for the $\vp$-Laplacian already considered in
\cite{RigoliSalvatoriVignati-Pacific}, \cite{RigoliSetti-Revista},
\cite{RigoliSalvatoriVignati-Revista} and \cite{PigolaRigoliSetti-Memoirs}.
It is worth pointing out that, besides allowing the presence of a term depending
on the gradient of $u$, we are able to deal with $C^1$ functions, removing the
requirement that $u\in C^2(M)$ and that the vector field
$|\nabla u|^{-1} \vp(|\nabla u|) \nabla u$ be $C^1$.

In order to formulate our version of the weak maximum principle, we note
that if $X$ is a $C^1$ vector field, and $v$ a positive continuous function on an open set $\Omega$,
then the following two statements
\begin{itemize}
\item[\rmi] $\inf_{\Omega} v^{-1} \diver X \leq C_o $,
\item[\rmii] if  $\diver X\geq  C v$ on $\Omega$ for some constant $C$, then
$C\leq C_o.$
\end{itemize}
Since \rmii is meaningful for in distributional sense, we may take it as
the weak definition of \rmi, and apply it to the case where $X$ is only $C^0$
($L^\infty_{loc}$ would suffice), and $v$ is only assumed to be non-negative
and continuous. Indeed, it is precisely the implication  stated in
\rmii that will allow us to prove Theorem~B.

In view of  applications to the case of the diffusion operator $L_{D,\vp}$,
it may also be useful to observe that, if the weight function $D(x)$ is assumed to be $C^1$
(indeed, $W^{1,1}_{loc}$ is enough if $X$ is assumed to be merely in
$L^\infty_{loc}$), then the weak inequality
$$
D(x)^{-1}\diver X  \geq C v
$$
is in fact equivalent to the inequality
$$
\diver X \geq C D(x) v
$$


%
%
%
\begin{theorem}
\label{improved weak max principle}
Let $(M, \langle\,,\rangle)$ be a complete Riemannian manifold,
let $D(x)\in C^0(M)$ be a positive weight on
$M$, and let $\vp$ satisfy ($\Phi_1$). Given $\sigma,$ $\mu$, $\chi\in \R$, let
\begin{equation*}
\eta = \mu + (\sigma -1)(1+\delta - \chi),
\end{equation*}
and assume that
\begin{equation*}
\sigma \geq 0, \quad \sigma -\eta \geq 0, \,\text{ and } \, 0\leq \chi < \delta.
\end{equation*}
Let $u\in C^1 (M)$ be a non constant function such that
\begin{equation}
\label{u growth bis}
\hat u = \limsup_{r(x) \to +\infty} \frac{u(x)}{r(x)^\sigma} <+\infty.
\end{equation}
and suppose that either
\begin{equation}
\label{vol growth exp bis}
\liminf_{r\to +\infty} \frac{\log \volD B_r}
{r^{\sigma -\eta}} = d_0
<+\infty\quad \text{ if } \sigma -\eta >0
\end{equation}
or
\begin{equation}
\label{vol growth poly  bis}
\liminf_{r\to +\infty} \frac{\log \volD B_r}{\log r}= d_0 <+\infty\quad \text{ if } \sigma -\eta =0.
\end{equation}
Suppose that $\gamma\in \R$ is such that the superset $\Omega _\gamma =\{x\in M \,:\, u(x)>\gamma\}$
is not empty, and that the weak inequality
\begin{equation}
\label{u inf bound}
\div{ D(x) |\nabla u|^{-1} \vp(|\nabla u|) \nabla u}
\geq
K   \bigl( 1+r(x)\bigr)^{-\mu} |\nabla u|^\chi D(x)
\end{equation}
holds on $\Omega_\gamma$. Then the constant $K$ satisfies
\begin{equation}
\label{K ineq} K \leq C(\sigma,\delta,\eta,\chi,d_0) \max\{\hat u,
0\}^{\delta-\chi}
\end{equation}
where    $C=C(\sigma,\delta,\eta, \chi,\d_0)$ is given
by
\begin{equation}
\label{C expression}
C=
\begin{cases}
0 &\text{if } \sigma = 0\\
A d_0(\sigma -\eta)^{1+\delta - \chi} &\text{ if }\sigma>0, \, \eta<0  \\
A d_0 \sigma^{\delta -\chi} (\sigma - \eta) &\text{ if } \sigma >0, \, \eta\geq 0,
\end{cases}
\end{equation}
if $\sigma -\eta > 0$  and by
\begin{equation}
\label{C expression bis}
C=
\begin{cases}
0 &\text{if } \sigma= 0 \\
 & \text{or   } \sigma >0, \, \delta(\sigma -1) +d_0-1\leq 0\\
A \sigma^{\delta - \chi} [\delta(\sigma -1) +d_0-1] & \text{if }
\sigma >0, \,\delta(\sigma -1) +d_0-1 > 0\\
\end{cases}
\end{equation}
if $\sigma -\eta = 0$.
\end{theorem}

\begin{remark}
{\rm
According to what observed before the statement, if $u$ in $C^2$,
the vector field $|\nabla u|^{-1} \vp\bigl(|\nabla u |\bigr) \nabla
u$ is $C^1$ and $\chi=0$, then the conclusion of the theorem is that
\begin{equation*}
\inf_{\Omega_\gamma} \bigl( 1+r(x)\bigr)^\mu
L_{D, \vp} u
\leq C(\sigma,\delta,\eta, \chi,\d_0) \max\{\hat u,0 \}^{\delta
},
\end{equation*}
and we recover an improved version of Theorem 4.1 in \cite{PigolaRigoliSetti-Memoirs}.
}
\end{remark}

\begin{proof}
The proof is an adaptation of that of Theorem 4.1 in
\cite{PigolaRigoliSetti-Memoirs}.
Clearly we may assume that $K>0$, for otherwise there is nothing to prove.

Note also that since $u$ is assumed to be non-constant, then
it cannot be constant on any connected component $E_o$ of
$\Omega_\gamma$. Indeed, if $u$ were constant in $E_o$, then
$\emptyset\ne\partial E_o\subseteq \partial \Omega_\gamma$.
Since, by continuity, $u=\gamma$ on $\partial \Omega_\gamma,$
we would conclude that $u\equiv \gamma$ on $E_o\subset \Omega_\gamma$,
contradicting the fact that $u>\gamma$ on $\Omega_\gamma.$
%
%
%

Next, because both the assumptions and the conclusions of the theorem are left unchanged
by adding a constant to $u$, arguing as in the proof of Theorem 4.1 in
\cite{PigolaRigoliSetti-Memoirs} shows that given $b> \max\{ \hat u, 0\}$, we may assume that
\begin{equation}
\label{u conditions}
\rmi \, \frac{u} {(1+r)^\sigma} < b \,\text{ and } \, \rmii
\, u (x_o)>0 \,
\text{ for some }\,  x_o\in \Omega_\gamma.
\end{equation}
Further, we observe that if (\ref{K ineq}) follows from  (\ref{u inf bound}) for some
$\gamma$ then the conclusion holds for any $\gamma'\leq \gamma$. Thus,  by increasing
$\gamma$ if necessary, we may also suppose that
$\gamma>0.$

%
%

We fix $\theta \in (1/2,1)$ and choose $R_0>0$ large enough
that $\modnabla u\not \equiv 0$ on  the non empty set
$B_{R_0}\cap \Omega_\gamma$.
Given $R>R_0$, let $\psi\in C^\ty(M)$ be a cut off function such that
\begin{equation}
\label{cutoff conditions}
0\leq \psi\leq 1, \quad \psi\equiv 1 \, \text{ on }\,
B_{\theta R}, \quad \psi \equiv 0 \, \text{ on }\,
M\setminus B_{R}, \quad \modnabla \psi \leq \frac C {
R(1-\theta)},
\end{equation}
for some absolute constant $C>0.$
Let also $\lambda \in C^1(\R)$ and
$F(v,r)$$\in C^1(\R^2)$ be such that
\begin{equation}
\label{lambda conditions}
0\leq \lambda \leq 1, \quad
\lambda =0 \,\text{ on }\, (-\ty, \gamma],\quad
\lambda >0,\, \, \lambda'\geq 0\, \text{ on }\, (\gamma,+\ty).
\end{equation}
and
\begin{equation}
\label{F definition}
F(v,r)>0, \quad \frac{\partial F}{\partial v}(v,r) <0
\end{equation}
on $[0,+\ty)\times [0.+\ty),$ where $v$ is given by
\begin{equation}
\label{v def}
v= \alpha (1+r)^\sigma - u.
\end{equation}
and $\alpha$ is a constant  greater than $b,$  so that $v>0$ on
$\Omega_\gamma.$
Indeed, according to (\ref{u conditions}), and the assumption that
$\gamma\geq 0,$  so that $u>0$ on $\Omega_\gamma,$ we have
\begin{equation}
\label{v ineq}
(\alpha - b) (1+r)^\sigma \leq v \leq \alpha
(1+r)^\sigma\quad\text{on }\,  \Omega_\gamma,
\end{equation}

By definition of the weak inequality (\ref{u inf bound}), for every non-negative
test function $0\leq \rho \in H^1_0(\Omega_\gamma)$,
\begin{equation*}
-\int_{\Omega_\gamma} \langle \nabla \rho, \nablaphi u \rangle
D(x)\, dx
\geq
K
\int_{\Omega_\gamma} \rho (1+r)^{-\mu} \modnabla u ^\chi D(x)\, dx.
\end{equation*}
We  use as test function the function $\rho= \psi^{1+ \delta} \lambda
(u) F(v,r)$ which is non-negative, Lipschitz, compactly supported in $M$
and vanishes on $M\setminus (\Omega_\gamma \cap B_R(o))$. Inserting
the expression for $\nabla \rho$ in the above integral inequality,
using the conditions $\lambda'>0,$ $F(v,r)>0$, $\partial F/\partial
v<0,$ and $\modnabla u \leq  A^{-1/\delta} \vp(\modnabla u)
^{1/\delta}$, which in turn follows from the structural condition
$\vp(t)\leq A t^\delta$,   after some computations we obtain
\begin{multline}
\label{int ineq}(1+\delta ) \int \psi^{\delta} \lambda (u) F(v,r)
 \vp(\modnabla u) \modnabla \psi D(x)\, dx
\\
\geq \int \psi^{1+\delta} \lambda (u) \left | \frac{\partial
F}{\partial v} \right | B(u, r) D(x)\, dx
\end{multline}
where
\begin{equation}
\label{B def}
\begin{split}
B(u,r) &=  A^{-1/\delta} \vp (\modnabla u)  ^{1+1/\delta} \\
& +
K A^{-\chi/\delta} \frac{F(v,r)}{|\partial F/\partial v|} (1+r)^{-\mu}
\vp (\modnabla u) ^{\chi/\delta}\\
&+
\left(\frac{\partial F/\partial r}{|\partial F/\partial v|} - \alpha\sigma (1+r)^{\sigma -1} \right)
\modnabla u^{-1} \vp(\modnabla u) \langle \nabla r, \nabla u\rangle.
\end{split}
\end{equation}

Now one needs to considers several cases separately. We treat in detail only the case
where $M$ satisfies the volume growth condition (\ref{vol growth exp bis}),
$\sigma>0,$ and $\eta< 0$.

In this case we let
\begin{equation*}
F(v,r) = \exp\bigl[-qv(1+r)^{-\eta}\bigr],
\end{equation*}
where $q>0$ is a constant that will be specified later. An
elementary computation which uses the estimate for $v$ given in
(\ref{v ineq})
shows that
\begin{align}
\label{partial F estimate 1}
&0\geq \frac{
\frac{\partial F}{\partial r}(v,r)
}
{
\left|\frac{\partial F}{\partial v}(v,r)\right|
}
- \alpha \sigma (1+r)^{\sigma -1}
\geq
-\alpha
(\sigma -\eta)
(1+r)^{\sigma -1}\\
\intertext{and}
\label{partial F estimate 2}
&
\frac {F(v,r)}
{
\left|\frac{\partial F}{\partial v}(v,r)\right|}
= \frac 1 q (1+r)^{\eta}
.
\end{align}
Inserting (\ref{partial F estimate 1}) and (\ref{partial F estimate 2})
into (\ref{B def}), and using the
Cauchy--Schwarz inequality we deduce that
\begin{multline}
\label{B estimate 1}
B(u, r)\geq
\vp( \modnabla u)^{\chi/\delta} \Bigl\{
\frac 1 {A^{1/\delta}}  \vp(\modnabla u)^{\frac{\delta + 1-\chi}\delta} +
\frac K{qA^{\chi/\delta}} (1+r)^{(1+\delta-\chi)(\sigma-1)}\\
- \alpha (\sigma -\eta)
(1+r)^{\sigma -1}
\vp(\modnabla u)^{\frac{\delta -\chi}\delta}
\Bigr \}.
\end{multline}
In order to estimate the right hand side of (\ref{B estimate 1})
we use the following calculus result (see \cite{PigolaRigoliSetti-Memoirs}, Lemma
4.2): let $\nu$, $\rho$, $\beta$, $\omega$ be positive constants, and
let $f$ be the function defined on $[0,+\ty)$ by
$f(s) =\omega s ^{1+\nu} + \rho - \beta s^\nu.$ Then the
inequality $f(s)\geq \Lambda s^{1+\nu}$ holds on $[0,+\ty)$ provided
\begin{equation}
\label{Lambda estimate}
\Lambda \leq \omega - \frac{\nu\beta^{1+1/\nu}}
{(1+\nu)^{1+1/\nu} \rho^{1/\nu}}.
\end{equation}
Applying this result with $\nu = \delta -\chi$ and  $s=\vp(\modnabla
u)^{1/\delta}$, and recalling the definition of $\eta$
we deduce that the estimate
\begin{equation}
\label{B estimate 2}
B( u, r)\geq \Lambda \vp(\modnabla u)^{1+1/\delta},
\end{equation}
holds provided
\begin{equation}
\label{Lambda estimate 2}
\Lambda \leq
\frac 1{A^{1/\delta}} -
\frac
{\nu q^{1/\nu} A^{\chi/\delta\nu}[\alpha
(\sigma -\eta)]^{1+1/\nu}}
{(1+\nu)^{1+1/\nu}K^{1/\nu}}.
\end{equation}
In particular, given $\tau\in(0,1)$  if we let
\begin{equation}
\label{Lambda q def}
\Lambda = \frac{1-\tau}{A^{1/\delta}}
\,\,\text{ and }\,\,
q= \frac{\tau^{\nu} (1+\nu)^{1+\nu}}
{\nu^\nu A [\alpha
(\sigma -\eta)
]^{1+\nu}}K
\end{equation}
then $\Lambda$ is positive, and satisfies
(\ref{Lambda estimate 2}) with equality.

Inserting  (\ref{B estimate 2}) and the expression for
$\partial F/\partial v$ into (\ref{int ineq}), we deduce that
\begin{multline*}
\frac{q\Lambda}{1+\delta}
\int_{\Omega_\gamma\cap B_R}
\psi^{1+\delta} \lambda (u) F(v,r) (1+r)^{-\eta}
\vp(\modnabla u)^{1+1/\delta}D(x) \, dx\\
\leq
\int_{\Omega_\gamma\cap B_R}
\psi^{\delta} \lambda (u) F(v,r) \modnabla \psi
\vp(\modnabla u)D(x)\,dx.
\end{multline*}
Now the proof proceeds as in \cite{PigolaRigoliSetti-Memoirs}:
applying H\"older inequality with conjugate exponents $1+\delta$
and $1+1/\delta$ to the integral on the right hand side, and
simplifying we  obtain
\begin{multline}
\label{int estimate 2}
\Bigl(\frac{q\Lambda}{1+\delta}\Bigr)^{1+\delta}
\int_{\Omega_\gamma\cap B_R}
\psi^{1+\delta} \lambda (u) F(v,r) (1+r)^{-\eta}
\vp(\modnabla u)^{1+1/\delta} D(x)\\
\leq
\int_{\Omega_\gamma\cap B_R}
\lambda (u) F(v,r)(1+r)^{\eta\delta}
 \modnabla \psi^{1+\delta}D(x).
\end{multline}
By the volume growth assumption (\ref{vol growth exp bis}), for every $d>d_0$, there
exists a diverging sequence $R_k\uparrow +\ty$ with $R_1>2R_0$ such that
\begin{equation}
\label{log vol estimate}
\log \vol B_{R_k} \leq d R_k^{\sigma -\eta}.
\end{equation}
Since  $\theta R_k>R_k/2>R_0,$  we may let $R=R_k$ in (\ref{int estimate 2}), and
use the support properties of $\psi$, the estimate for  $\modnabla \psi$, and the fact that
$\lambda \leq 1$, $\eta<0$ to show that
\begin{multline}
\label{E estimate}
E=
\Bigl(\frac{q\Lambda}{1+\delta}\Bigr)^{1+\delta}
\int_{\Omega_\gamma\cap B_{R_0}}\!\!\!
\lambda (u) F(v,r)\vp(\modnabla u)^{1+1/\delta} D(x)\\
\leq
 C^{1+\delta}
 (1+\theta R_k)^{\eta\delta}
[(1-\theta)R_k]^{-(1+\delta)}
\int_{\Omega_\gamma\cap (B_{R_k}\setminus B_{\theta R_k})}
F(v,r) D(x).
\end{multline}
Now, since $\modnabla u \not \equiv 0$ on
$\Omega_\gamma\cap B_{R_0}$, then $E>0$. On the other hand, using
the bound (\ref{v ineq}) for $v,$ and the expression of $F$ we get
\begin{equation*}
F(v,r)\leq \exp\bigl(- q(\alpha -b) (1+\theta R_k)^{\sigma -\eta}
\bigr)
\end{equation*}
on $\Omega_\gamma\cap (B_{R_k}\setminus B_{\theta R_k})$, so
inserting this
into the right hand side of (\ref{E estimate})
we conclude that
\begin{multline}
\label{E estimate bis}
0<E\leq C  R_k ^{\delta\eta -(1+\delta)}\\
\times  \exp\bigl(d R_k^{\sigma -\eta}
-q(\alpha -b) (1+\theta R_k)^
{\sigma -\eta}
\bigr),
\end{multline}
where $C$ is a constant independent of $k$. In order for this
inequality to hold for every $k$, we must have
\begin{equation*}
d\geq (\alpha -b) q \theta^
{\sigma -\eta},
\end{equation*}
whence, letting $\theta$ tend to $1$,
\begin{equation*}
d\geq (\alpha -b)q.
\end{equation*}
We set $\alpha = t b$, insert the definition (\ref{Lambda q def})
of $q$ in the above inequality, solve with respect to $K$,
and then let $\tau$ tend to $1$  to obtain
\begin{equation*}
K\leq A d  b ^\nu
(\sigma -\eta)^{1+\nu}
\frac{\nu^\nu}{(1+\nu)^{1+\nu}}
\frac{t^{1+\nu} }{t-1}
.
\end{equation*}
The conclusion is then obtained minimizing with respect to $t>1$, letting $d\to d_0$
and $b\to \max\{\hat u, 0\}$ and recalling that $\nu=\delta-\chi$.

The other cases are treated adapting  the arguments carried out in the proof of
\cite{PigolaRigoliSetti-Memoirs} Theorem 4.1, cases II and III, and of Theorem 4.3 for the case of
polynomial volume growth.
\end{proof}

\begin{proof}[Proof of Theorem~B]
We begin by showing that if under the assumptions of the theorem,
$u$ is necessarily bounded above. Indeed, assume by contradiction that
$u^*=+\infty$, so that, by (\ref{u_growth}), $\sigma>0$,
and there exists $\gamma_o$ and $C>0$ such that $f(t)>C $ for $t\geq \gamma$.
Keeping into account the assumptions on $b$ and $l$, we deduce that $u$ satisfies the
differential inequality
\begin{equation*}
\div{ D(x) |\nabla u|^{-1} \vp(|\nabla u|) \nabla u}
\geq
K   \bigl( 1+r(x)\bigr)^{-\mu} |\nabla u|^\chi D(x)
\end{equation*}
weakly on $\Omega_{\gamma_o}$, with a constant $K>0.$ On the other
hand, because of  growth assumption on $u$, the constant
$\hat u$ in the statement
of Theorem~\ref{improved weak max principle} is equal to zero, and this shows that
$K=0$, and the contradiction shows that $u^*<+\infty$ is bounded above.

Assume now that $f(u^*)>0.$ Since $f(t)>0$ for $t>0,$ by continuity there exists
$\gamma_o$ such that $f(u)\geq C>0$ on $\Omega_{\gamma_o}$, and a contradiction is
reached as above.

The final statement follows immediately from this and from the
assumptions.
\end{proof}

\section{Proof of Theorem C}
The aim of this section is to prove Theorem C in the Introduction
together with a version covering the case of the mean curvature
operator. Before proceeding, we analyze the Keller--Osserman
condition
\begin{equation}
\tag{$\rho$KO}
{\ds\frac{e^{\int_0^t \rho(z) dz}}{K^{-1}\bigl(\hat F(t)\bigr)}}\in
L^ 1(+\infty),
\end{equation}
where $\rho \in C^0(\R^+_0)$,  is non-negative on $\R_0^+$ and
$\hat{F}(t)= F_{\rho, \omega}$ is defined in (\ref{F_rho_omega}),
namely,
\begin{equation}
\tag{\ref{F_rho_omega}}
F_{\rho,\omega} (t) = \int_0^t f(s) e^{(2-\omega) \int_0^s \rho(z)
dz} ds.
\end{equation}

\begin{lemma}
\label{lemma6.1} Assume that ($F_1$) ($L_1$) and the first part of
($\theta$)$_1$ with $\theta <2$ hold, and let $\omega=\theta$
and $\sigma \in \R^+$. Then ($\rho$KO) is equivalent to
\begin{equation}
\tag{$\rho$KO$_\sigma$}
{
\ds\frac{e^{\int_0^t \rho(z) dz}}{K^{-1}\bigl(\sigma \hat F(t)\bigr)}
}
\in
L^ 1(+\infty).
\end{equation}
\end{lemma}

\begin{proof}
Assume first that $\sigma\leq 1$. Since $K^{-1}$ is non-decreasing,
\begin{equation*}
\frac 1 {K^{-1}(\hat F (t))}\leq \frac 1 {K^{-1}(\sigma \hat F (t))}
\end{equation*}
and ($\rho$KO$_\sigma$) implies ($\rho$KO). On the other hand,
according to Proposition~\ref{prop 3.2bis} and Remark~\ref{rmk prop
3.2bis} there exists a constant $B\geq 1$ such that
\begin{equation*}
\frac{\sigma^{1/(2-\theta)}}{K^{-1}(\sigma \hat F(t))}\leq
\frac B{K^{-1}(\hat F(t))}
\quad \text{ on } \R^+,
\end{equation*}
and ($\rho$KO) implies ($\rho$KO$_\sigma$). Thus the stated
equivalence holds when $\sigma\leq 1$. Then the case $\sigma\geq 1$
follows as in Lemma~\ref{lemma_KOsigma}.
\end{proof}

We observe that in favorable circumstances (KO) and ($\rho$KO) are
indeed equivalent. For instance we have

\begin{proposition}
\label{prop6.1}
Assume that (F$_1$), (L$_1$), ($\vp\ell$)$_2$ and
($\rho$) hold. If $\rho \in L^1(+\infty)$ and $\omega\leq 2$ then
($\rho$KO) is equivalent to (KO).
\end{proposition}

\begin{proof}
Observe first of all that since $0\leq \rho\in L^1((0,+\infty))$
($\rho$KO) is equivalent to
\begin{equation}
\label{eq6.2}
\frac 1 {K^{-1} (\hat F(t))}\in L^1(+\infty).
\end{equation}
Since  $\omega \leq 2$ we also have
\begin{equation*}
1\leq e^{(2-\omega)\int_0^s \rho(z)dz }\leq \Lambda,
\end{equation*}
and therefore
\begin{equation}
\label{eq6.3}
F(t) = \int_0^t f(s) ds \leq \hat F(t) = \int_0^t f(s) e^{(2-\omega)\int_0^s \rho(z)dz }
\leq \Lambda F(t).
\end{equation}
Recalling that $K^{-1}$ increasing, the left hand side inequality in
(\ref{eq6.3}) shows that
\begin{equation*}
\int^{+\infty} \frac {dt} { K^{-1}(\hat F(t))} \leq
\int^{+\infty} \frac {dt} {K^{-1}( F(t))}
\end{equation*}
and, by (\ref{eq6.2}), (KO) implies ($\rho$KO).

On the other hand, since,  by ($F_1$), $f$ is $C$-increasing with constant $C\geq 1$,
so is also the integrand in the definition of $\hat F$, and therefore the right hand
side inequality inequality in (\ref{eq6.3}) and the argument in the proof of
Lemma~\ref{lemma_KOsigma}, with $\sigma=\Lambda^{-1}$ and $F$
replaced by $\hat F$, show that
\begin{equation}
\label{eq6.4}
\int^{+\infty} \!\!\!\!\!\frac { ds} {K^{-1}(F(s))}
\leq \int^{+\infty} \!\!\!\!\!\frac { ds} {K^{-1}(\Lambda ^{-1} \hat
F(s))}
\leq  C\Lambda \int^{+\infty} \!\!\!\!\!\frac { dt} {K^{-1}( \hat F(
t))},
\end{equation}
and, again by (\ref{eq6.2}), ($\rho$KO) implies (KO).
\end{proof}

\begin{remark}
{\rm The above proposition generalizes Proposition 6.1 in
\cite{MagliaroMarimastroliaRigoli-Heisenberg}.
}
\end{remark}

\begin{proposition}
\label{prop6.5}
Assume that ($\Phi_0$), ($F_1$), ($L_1$), ($L_2$), ($\vp\ell$)$_1$,
($\theta$), (b), ($\rho$) and ($\rho$KO) with $\omega=\theta$
hold. Let $A>0$, $\beta\geq -2$, and, if $\lambda >0$ and $\theta$ are
the constants in (b) and ($\theta$), suppose that $\theta\leq 1 $
and
\begin{equation}
\label{eq6.6}
\begin{cases}
\lambda \geq 1 \qquad\quad\quad t^{\beta/2} \tilde b(t) ^{-1} \int_1^{t}
\tilde
b(s)^\lambda ds \leq C \quad\forall t\geq 1 &\text{if } \theta =1\\
\lambda (2-\theta)\geq  1
\quad t^{\beta/2} \tilde b(t) ^{\lambda(1-\theta) -1}
\leq C \quad \quad\quad\,\forall t\geq 1 &\text{if } \theta <1,
\end{cases}
\end{equation}
for come constant $C>0.$ The there exists $T>0$ sufficiently large such that,
for every $T\leq t_0<t_1$ and $0<\epsilon<\eta$, there exist $\bar T>t_1$ and a $C^2$
function $\alpha : [t_0, \bar T)\to [\epsilon, +\infty)$ which
is a solution of the problem
\begin{equation}
\label{eq6.7}
\begin{cases}
\vp'(\alpha ') \alpha '' + A t^{\beta/2} \vp(\alpha ') \leq \tilde
b(t) f(\alpha) \ell (\alpha)-\rho(\alpha)\vp'(\alpha') (\alpha')^2
 \quad \text{on }\, [t_0, \bar T)&\\
\alpha '>0 \text{ on } [t_0, \bar T), \,\, \alpha(t_0) = \epsilon,
\, \, \alpha(t)\to +\infty \text{ as } t\to \bar T^-&
\end{cases}
\end{equation}
and satisfies
\begin{equation}
\label{eq6.8}
\epsilon \leq \alpha \leq \eta \quad \text{ on } \, [t_0, t_1].
\end{equation}
\end{proposition}

\begin{proof}
The proof is a modification of that of Proposition~\ref{prop3.7}
so we only sketch it.

Note that since ($\theta$)$_1$ holds with $\theta\leq 1$, so does
($\vp\ell$)$_2$. Thus $K$ defines a $C^1$ diffeomorphism of  $\R^+_0$
and condition ($\rho$KO) is meaningful.

As in the proof of Proposition~\ref{prop3.7},  we may assume that $\tilde  b\leq 1$
for $t$ large. Choose $T>0$ large enough that $\tilde b'(t)\leq 0$  and $0<\tilde b(t)\leq 1$
in $[T,+\infty),$  let $t_0,$ $t_1$, $\epsilon$, $\eta$ as in the statement,
use Lemma~\ref{lemma6.1}, (b) and condition ($\rho$KO),   to define
$T_\sigma$ by means of the formula
\begin{equation*}
\int_{t_0}^{T_\sigma} \tilde b(s)^\lambda ds  = \int_\epsilon^{+\infty} \frac{e^{\int_0^s
\rho}}{K^{-1}(\sigma \hat F(s))},
\end{equation*}
and choose $\sigma\in (0,1]$ small enough to guarantee that $T_\sigma
>t_1.$

Next let $\alpha :[t_0, T_\sigma )\to [\epsilon, +\infty)$ be
defined by the formula
\begin{equation*}
\int^{T_\sigma}_{t} \tilde b(s)^{\lambda} ds =
\int_{\alpha(t)}^{+\infty}\frac{e^{\int_0^s
\rho}}{K^{-1}(\sigma \hat F(s))},
\end{equation*}
so that
\begin{equation*}
\alpha(t_0) = \epsilon,\quad \text{and}\quad \alpha(T_\sigma^-)= +\infty.
\end{equation*}
Differentiating we obtain
\begin{equation*}
\alpha' = \tilde b^\lambda K^{-1}(\sigma \hat
F)e^{-\int_0^\alpha\rho},
\end{equation*}
so that $\alpha'>0$, and rearranging, differentiating once
again, and simplifying we obtain,
\begin{equation}
\label{eq6.10}
\sigma f(\alpha) e^{(2-\theta)\int_0^\alpha \rho}
= \left(\frac{ e^{\int_0^\alpha\rho}}{\tilde
b^\lambda}\right) \frac{\vp'\left(\frac{\alpha' e^{\int_0^\alpha\rho}}{\tilde
b^\lambda}\right)}{\ell\left(\frac{\alpha' e^{\int_0^\alpha\rho}}{\tilde
b^\lambda}\right)} \left(\frac{\alpha' e^{\int_0^\alpha\rho}}{\tilde
b^\lambda}\right)',
\end{equation}
so that, in particular, $(\alpha' e^{\int_0^\alpha\rho}/\tilde
b^\lambda)'>0$.

We use the fact that $e^{\int_0^\alpha \rho}/\tilde b\geq 1$ to apply ($\theta$)$_1$,
we expand the derivative of $(\alpha' e^{\int_0^\alpha\rho}/\tilde
b^\lambda)$,  use  $\tilde b'\leq 0$, and rearrange
to obtain
\begin{equation}
\label{eq6.11}
\vp'(\alpha') \alpha'' \leq C \sigma f(\alpha)\ell(\alpha')\tilde
b^{\lambda (2-\theta)} - \rho(\alpha)\vp'(\alpha')^2.
\end{equation}
On the other hand,  we rewrite (\ref{eq6.10}) in the form
\begin{equation*}
\vp'\left(\frac{\alpha' e^{\int_0^\alpha\rho}}{\tilde
b^\lambda}\right)\left(\frac{\alpha' e^{\int_0^\alpha\rho}}{\tilde
b^\lambda}\right)' =
\sigma \tilde b^\lambda f(\alpha) \ell\left(\frac{\alpha' e^{\int_0^\alpha\rho}}{\tilde
b^\lambda}\right) e^{(1-\theta)\int_0^\alpha \rho},
\end{equation*}
integrate between $t_0$ and $t,$ and use the $C$-monotonicity of
$f$ and $\ell$ and ($\theta$)$_2$ to obtain
\begin{equation*}
\vp\left(\frac{\alpha' e^{\int_0^\alpha\rho}}
{\tilde b^\lambda}\right) -
\vp\left(\frac{\alpha' e^{\int_0^\alpha\rho}}{\tilde
b^\lambda}\right)(t_0) \leq C\sigma
f(\alpha)e^{(1-\theta)\int_0^\alpha \rho}
\ell\left(\frac{\alpha' e^{\int_0^\alpha\rho}} {\tilde b^\lambda}\right)
\int_0^t \tilde b^\lambda,
\end{equation*}
whence, rearranging and using the $C$-monotonicity of
$t^{\theta-1}\vp(t)/\ell(t)$,  $f$ and $\ell$, and the $\theta\leq
1$ shows that (see the argument that led to (\ref{eq3.19}) in the
proof of Proposition~\ref{prop3.7}
\begin{equation}
\label{eq6.12}
\begin{split}
\frac{\vp(\alpha')}{\ell(\alpha')} &\leq C \left(\frac{ e^{\int_0^\alpha\rho}}
{\tilde b^\lambda}\right)^{\theta-1} \frac{\vp\left(\frac{\alpha' e^{\int_0^\alpha\rho}}
{\tilde b^\lambda}\right)}{\ell\left(\frac{\alpha' e^{\int_0^\alpha\rho}}
{\tilde b^\lambda}\right)}\\
&\leq C\tilde b f(\alpha) \Biggl\{\sigma \tilde b^{\lambda
(1-\theta)-1} \int _0^t\tilde b^{\lambda} +
\frac{\vp\left(\frac{\alpha' e^{\int_0^\alpha\rho}}
{\tilde b^\lambda}\right)(t_0)}{f(\epsilon)\ell\left(\frac{\alpha' e^{\int_0^\alpha\rho}}
{\tilde b^\lambda}\right)}(t_0) \Biggl\}.
\end{split}
\end{equation}
Thus, combining (\ref{eq6.11}) and (\ref{eq6.12}) and arguing as in
Proposition~\ref{prop3.7} we deduce that
\begin{equation*}
\vp'(\alpha')\alpha'' + At{^\beta/2} \vp(\alpha') \leq N(\sigma)
\tilde b f(\alpha)\ell(\alpha') - \rho(\alpha) \vp'(\alpha')
(\alpha')^2
\end{equation*}
with
\begin{equation*}
\begin{split}
N_\sigma (t) &=
C \sigma \tilde b^{\lambda (2-\theta) - 1}
+  AC t^{\beta/2} \tilde b^{\lambda (1-\theta) -1}
\frac{\vp(K^{-1}(\sigma F(\epsilon)))}{\ell(K^{-1}(\sigma F(\epsilon)))f(\epsilon)}
\\
&+ AC \sigma t^{\beta/2} \tilde b^{\lambda (1-\theta) -1}\int_{t_0}^t\tilde
b(s)^\lambda.
\end{split}
\end{equation*}
The proof now proceeds exactly as in Proposition~\ref{prop3.7}.
\end{proof}

The next result is the analogue of Theorem~\ref{thm3.21} and Theorem~C in the Introduction
follows from it using Remark~\ref{proof ThmA}.

\begin{theorem}
\label{thm6.16} Let $(M\langle \,,\rangle)$ be a complete manifold
satisfying
\begin{equation}
\label{eq6.16'}
 \mathrm{Ricc}_{n,m} (L_D) \geq H^2(1+r^2)^{\beta/2},
\end{equation}
for some $n>m$, $H>0$ and $\beta\geq -2$, and assume that (h), (g),
($\rho$), ($\Phi_0$), ($F_1$), ($L_1$) ($L_2$), ($\vp\ell$)$_1$ and
($\theta$) hold. Let also $b(x)\in C^0(M)$ be strictly positive on
$M$ and such that
\begin{equation}
\label{eq6.17}
b(x)\geq \tilde b(r(x)) \quad \text{ for }\, r(x)\gg
1,
\end{equation}
with  $\tilde b$ satisfying (b), and (\ref{eq6.6}). Finally, suppose
that ($\rho$KO) holds with $\omega=\theta$ in the definition of $\hat
F$. Then any entire classical weak solution of the differential
inequality
\begin{equation}
\tag{\ref{ineq_minus}} L_{D, \vp} u \geq b(x) f(u) \ell(\modnabla u)
- g(u) h(\modnabla u),
\end{equation}
is either non-positive or constant. Moreover, if $u\geq 0$ and
$\ell(0)>0$, then $u\equiv 0$.
\end{theorem}

\begin{proof}
The proof is modeled on that of Theorem ~\ref{thm3.21}. However, in
the case where $u$ is bounded above, in order to prove that, if $u$
takes on positive values and is non-constant then
\begin{equation*}
u^*_o=\sup_{B_{r_o}} u < \sup u = u^*,
\end{equation*}
we argue as follows. Assume that $u$ attains its supremum  $u^*>0$
and let $\Gamma = \{x : u(x) =u^*\}$. Clearly $\Gamma$ is closed
and nonempty. We are going to show that it is also open so, by
connectedness, $\Gamma = M$ and $u$ is constant. To this end, let
$x_o\in \Gamma$. We have $b(x)f(u)\geq \frac 1 2 b(x_o)f(u^*)>0$
and $g(u)\leq 2C\rho(u^*)$ in a suitable neighborhood $U$ of
$x_o$. Moreover, by ($\theta$)$_1$ and $(h)$, we may estimate
\begin{equation*}
h(s)\leq C s^2 \vp'(s)\leq C\frac{\vp'(1)}{\ell(1)}  s^{2-\theta}\ell(s)
= C s^{2-\theta} \ell (s),
\quad \forall s\leq 1,
\end{equation*}
so that, in  $U$,
\begin{equation*}
b(x) f(u) \ell (|\nabla u|) - g(u) h(|\nabla u|)\geq
\ell(|\nabla u|)
\Big(\frac{b(x_o)}{2} f(u^*) - C\rho(u^*) |\nabla
u|^{2-\theta}
\Big).
\end{equation*}
Since $\nabla u(x_o)=0$ it is now clear that there exists a
neighborhood $U'\subset U$ of $x_o$ where the right hand side the
above inequality is non-negative. Thus,
\begin{equation*}
L_{D,\vp} u\geq 0 \quad \text{ in }\,\, U'
\end{equation*}
and  $u=u^*$ in $U'$ by the strong maximum principle.

We note in passing that if $\ell(0)>0$ the required conclusion may be
obtained without having to appeal to condition ($\theta$)$_1$.

The rest of the proof proceeds as in
Theorem~\ref{thm3.21} using Proposition~\ref{prop6.5} instead of
Proposition~\ref{prop3.7}.
 \end{proof}
As we did for Theorem~\ref{thm3.21} in
Section~\ref{section_proof_ThmA},  even in this case we can provide
a version of the above result valid for a class of operators which
include the mean curvature operator. In order to do this we need to
introduce the appropriate Keller-Osserman condition. Given
$\omega\in \R$, let $\rho$ satisfy ($\rho$) and let $\hat F$ be
defined in (\ref{F_rho_omega}). We assume
($\vp\ell$)$_3$ holds and let $\hat K$ be defined in (\ref{eq4.0}).
The version of Keller-Osserman condition we consider is then
\begin{equation}
\label{rho_hatKO}\tag{$\rho\hat KO$} \frac{e^{\int_0^t \rho}}{\hat
K^{-1}\bigl(\hat F (t)\bigr)}\in L^1(+\infty).
\end{equation}

Modifications of the arguments of
Section~\ref{section_further_version} allow to obtain the following

\begin{theorem}
\label{thm6.20} Let  $(M\langle \,,\rangle)$ be a complete manifold
satisfying (\ref{eq6.16'})
for some $n>m$, $H>0$ and $\beta\geq -2$, and assume that (h),
(g), ($\rho$), ($\Phi_0$), ($F_1$), ($L_1$) ($L_2$),
($\vp\ell$)$_1$ and ($\theta$) hold. Let also $b(x)\in C^0(M)$
be strictly positive on $M$ and satisfying (\ref{eq6.17})
with  $\tilde b$ satisfying (b), and (\ref{eq6.6}). Finally, suppose
that ($\rho$\^KO) holds with $\omega=\theta$ in the definition of
$\hat F$. Then any entire classical weak solution of the
differential inequality
\begin{equation}
\tag{\ref{ineq_minus}} L_{D, \vp} u \geq b(x) f(u) \ell(\modnabla u)
- g(u) h(\modnabla u),
\end{equation}
is either non-positive or constant. Moreover, if $u\geq 0$ and
$\ell(0)>0$, then $u\equiv 0$.
\end{theorem}
We leave the details to the interested reader, and  merely point
out that, according to what remarked in the proof of
Theorem~\ref{thm6.16}, if $\ell(0)>0$ then  it suffices to assume
($\theta$)$_2$ in the statement of  Theorem~\ref{thm6.20}.

\end{document}